\documentclass{amsart}

\usepackage{amsmath}
\usepackage{amssymb,mathrsfs}
  \usepackage{paralist}
  \usepackage{graphics} 
  \usepackage{epsfig} 
\usepackage[colorlinks=true]{hyperref}
\hypersetup{urlcolor=blue, citecolor=red}

\newtheorem{thm}{Theorem}[section]
\newtheorem{cor}{Corollary}
\newtheorem{lem}[thm]{Lemma}
\newtheorem{prop}{Proposition}

\newcommand{\B}{\mathcal{B}}
\newcommand{\D}{\mathcal{D}}
\newcommand{\E}{\mathcal{E}}

\newcommand{\LL}{\mathcal{L}}
\def\cL{\LL}
\newcommand{\HH}{\mathcal{H}}

\newcommand{\cS}{\mathcal{S}}
\def\cT{\mathcal{T}}

\newcommand{\R}{\mathbb{R}}

\newcommand{\Sm}{\mathscr{S}}

\newcommand{\al}{\alpha}

\newcommand{\ga}{\gamma}
\newcommand{\de}{\delta}
\newcommand{\e}{\varepsilon}
\newcommand{\fy}{\varphi}
\newcommand{\om}{\omega}
\newcommand{\la}{\lambda}

\newcommand{\s}{\sigma}
\newcommand{\ta}{\tau}
\newcommand{\ka}{\kappa}

\newcommand{\rh}{\rho}

\newcommand{\De}{\Delta}
\newcommand{\Om}{\Omega}

\newcommand{\La}{\Lambda}

\newcommand{\p}{\partial}
\newcommand{\na}{\nabla}

\newcommand{\supp}{\operatorname{supp}}
\newcommand{\Cu}{\bigcup}

\newcommand{\lec}{\lesssim}
\newcommand{\gec}{\gtrsim}

\newcommand{\IN}[1]{\text{ in }#1}

\newcommand{\I}{\infty}

\newcommand{\ti}{\widetilde}

\newcommand{\U}{\underline}
\newcommand{\LR}[1]{{\langle #1 \rangle}}

\newcommand{\EQ}[1]{\begin{equation}\begin{split} #1 \end{split}\end{equation}}
\newcommand{\Del}[1]{}
\newcommand{\CAS}[1]{\begin{cases} #1 \end{cases}}
\newcommand{\mat}[1]{\begin{pmatrix} #1 \end{pmatrix}}
\newcommand{\pt}{&}
\newcommand{\pr}{\\ &}
\newcommand{\pq}{\quad}
\newcommand{\pn}{}

\def\ti{\tilde}

\numberwithin{equation}{section}

\def\uu{\underline{u}}
\def\uv{\underline{v}}
\def\uN{\underbar{N}}
\def\uga{\underline{\ga}} 
\def\ug{\underline{g}}
\def\vLa{\vec\La\,}

\def\diag{\mathrm{diag}}

\def\eps{\varepsilon}
\def\nn{\nonumber}
\def\weakto{\rightharpoonup}
\def\embed{\hookrightarrow}
\def\dist{\mathrm{dist}}
\def\sg{\mathfrak{s}}
\def\lam{\lambda}
\newcommand{\HL}{\mathfrak{H}}
\newcommand{\Sg}{\mathfrak{S}}
\newcommand{\sign}{\operatorname{sign}}
\newcommand{\Ext}{E_{\operatorname{ext}}}
\newcommand{\V}{\vec}

\newcommand{\CE}{\mathfrak{C}}

\title[Nonradial critical wave equation]{Global dynamics of the nonradial energy-critical wave equation above the ground state energy}

\author[Joachim Krieger, Kenji Nakanishi and Wilhelm Schlag]{}



\email{joachim.krieger@epfl.ch}
\email{n-kenji@math.kyoto-u.ac.jp}
\email{schlag@math.uchicago.edu}

\thanks{Support of the National Science Foundation  DMS-0617854 for the third author, and  the Swiss National Fund for
the first author  are gratefully acknowledged. The authors thank Thomas Duyckaerts and Carlos Kenig for their interest in this work, as well as for helpful conversations.  }

\begin{document}
\maketitle

\centerline{\scshape Joachim Krieger }
\medskip
{\footnotesize
 \centerline{B\^{a}timent des Math\'ematiques, EPFL}
\centerline{Station 8,
CH-1015 Lausanne,
  Switzerland}
} 

\medskip

\centerline{\scshape Kenji Nakanishi}
\medskip
{\footnotesize
 \centerline{Department of Mathematics, Kyoto University}
\centerline{Kyoto 606-8502, Japan}
} 

\medskip

\centerline{\scshape Wilhelm Schlag}
\medskip
{\footnotesize
 \centerline{Department of Mathematics, The University of Chicago}
\centerline{5734 South University Avenue, Chicago, IL 60615, USA}
} 

\bigskip

 \centerline{(Communicated by Frank Merle)}

\begin{abstract}
In this paper we establish the existence of certain classes of solutions to the energy critical nonlinear wave equation in dimensions $3$ and~$5$
assuming that the energy exceeds the ground state energy only by a small amount.  No radial assumption is made.
 We find that there exist four sets in $\dot H^{1}\times L^{2}$ with nonempty interiors which correspond to all possible combinations of finite-time blowup on the one hand, and global existence and scattering to a free wave, on the other hand, as $t\to\pm\I$.
\end{abstract}

\section{Introduction}
Consider the energy-critical nonlinear wave equation with real-valued $u$
\EQ{\label{main PDE}
 \ddot u -\De u = |u|^{p-1}u, \pq u(t,x):\R^{1+d}\to\R, \pq p=\frac{d+2}{d-2}=2^*-1, \pq d=3 \text{ or } 5,}
in the energy space
\EQ{\label{eq:uvec}
 \vec u(t):=(u(t),\dot u(t)) \in \HH:=\dot H^1(\R^d)\times L^2(\R^d),}
which is the real Hilbert space with the inner product
\EQ{
 \LR{\uu,\uv}_\HH:=\LR{\na u_1|\na v_1}+\LR{u_2|v_2}, \pq \LR{f|g}:=\int_{\R^d}f(x)\cdot g(x)\,dx.}
Henceforth $\vec u=(u,\dot u)$ denotes the vector derived from a time function $u(t)$, while a general vector is denoted like $\uu=(u_1,u_2)$.
The seminorm on any domain $\Om\subset\R^d$ is defined by
\EQ{
 \|\uu\|_{\HH(\Om)}^2:=\|\na u_1\|_{L^2(\Om)}^2+\|u_2\|_{L^2(\Om)}^2.}
We remark that the dimensional restriction $d=3$ or $5$ is needed only for using the blow-up characterization by Duyckaerts-Kenig-Merle~\cite{DKM2}.

Equation~\eqref{main PDE} is locally well-posed for $\vec u(0)\in\HH$, globally for small data, and may blow up in finite time (for example, for data of negative energy).  Moreover, $I=[0,T_{*})$ is a finite maximal time of existence if and only if
\EQ{
 \| u\|_{L^{q}_{t}(I; L^{q}_{x}(\R^{d}))} =\I,\qquad q:=\frac{2(d+1)}{d-2}.}
For a comprehensive review of these basic issues we refer the reader to~\cite{KM2}.

In a previous paper~\cite{CNLW1} the authors studied the global dynamics of radial solutions to~\eqref{main PDE}.
To state that result as well as the main result of this paper, we recall some of the basic structures associated with the critical equation.
First, one has the conserved energy of \eqref{main PDE}
\EQ{
 E(\vec u):=\int_{\R^d} \Bigl[\frac{|\dot u|^2+|\na u|^2}{2}-\frac{|u|^{2^*}}{2^*}\Bigr]\, dx}
 as well as the conserved momentum
\EQ{
 P(\vec u):=\LR{\dot u| \na u}.  }

Remarkably, \eqref{main PDE} admits the static Aubin solutions of the form
\EQ{
 W_\s=S_{-1}^\s W, \pq W(x) = \left[1+\frac{|x|^2}{d(d-2)}\right]^{1-\frac{d}{2}},}
where $S_{-1}^\s$ denotes the $\dot H^1$ preserving dilation
\EQ{
 (S_{-1}^\s\fy)(x)=e^{(d/2-1)\s}\fy(e^\s x).}
These are positive radial solutions of the static equation
\EQ{
 -\Delta W -|W|^{2^*-2}W=0,}
which are unique, up to dilation and translation symmetries, amongst the non-negative, non-zero (not necessarily radial) $C^{2}$ solutions, see~\cite{CGS}. They also minimize the static energy
\EQ{ \label{def J}
 J(\fy) := \int_{\R^{d}} \Bigl[ \frac12 |\nabla \fy|^{2} - \frac{1}{2^*} |\fy|^{{2^*}} \Bigr] \, dx,}
among all non-trivial static solutions.
The work of Kenig, Merle~\cite{KM1,KM2} and Duyckaerts, Merle~\cite{DM1,DM2} allows for a characterization of the global-in-time behavior of solutions with $E(\vec u)\leq J(W)$.

In this paper we  study the behavior of solutions with
\EQ{
 E(\vec u)<\sqrt{J(W)^2+\eps^4+|P(\vec u)|^2},
}
for some small $\eps>0$.
Solutions of subcritical focusing NLKG and NLS equations with radial data in~$\R^{3}$ of energy slightly above that of the ground
state were studied by the latter two authors in~\cite{NakS0,NakS3}. The nonradial subcritical Klein-Gordon equation in three dimensions was treated in~\cite{NakS2}.
The key feature of \eqref{main PDE} by contrast to NLKG is the scaling invariance of~\eqref{main PDE} manifested by
\EQ{
 u(t, x) \mapsto e^{\s(d/2-1)}u(e^\s t, e^\s x) = S_{-1}^\s u(e^\s t),
}
which leaves the energy unchanged.
In particular, the analogue of the ``one pass theorem'' proved in \cite{NakS} needs to be modified, specifically by replacing the discrete set of attractors $\{Q, -Q\}$ there by a $(2d+1)$--parameter family of solitons. For any $(\s,p,q)\in\R\times \R^d\times\R^d$, denote the scaling-Lorentz transform of $W$ by
\EQ{ \label{soliton family}
 W_\s(p,q)=W_\s(x-q+p(\LR{p}-1)|p|^{-2}p\cdot(x-q))}
where $\LR{p}:=\sqrt{1+|p|^2}$. Then for any fixed $(p,q)\in\R^{2d}$,
\EQ{
 u(t,x)=W_\s(p,q+tp/\LR{p})}
gives a ground state soliton of \eqref{main PDE}. Hence the ground state soliton family is
\EQ{
 \Sm:=\{(W_\s(p,q),-\na W_\s(p,q)\cdot p/\LR{p})\mid (\s,p,q)\in\R^{1+2d}\} \subset \HH.}

Note that in the subcritical NLS case \cite{NakS2}, the scaling parameter $\s$ is essentially fixed or at least bounded from above and below by the $L^2$ conservation law, but in the critical case there is no factor which a priori prevents the scale from going to $0$ or $+\I$.
On the other hand, by using the Lorentz transform, we can reduce the problem to the case of zero momentum, where the soliton family is
\EQ{
 \Sm_0:=\{\vec W_\s(x-q) \mid (\s,q)\in\R^{1+d}\}\subset\HH, \pq \vec W_\s=(W_\s,0),}
with the energy constraint $E(\vec u)<J(W)+\e^2$ (slightly changing $\e>0$).

Introduce the ``virial functional''
\EQ{
 K(\fy): = \int_{\R^d}[|\nabla\fy|^2 - |\fy|^{2^*}]\,dx}
 and note that $K(W)=0$.
The following positivity is crucial for the variational structure around $W$
\EQ{
  \|\na\fy\|_2^2/d = J(\fy)-K(\fy)/2^*.}
Note that the derivative of $J(\fy)$ with respect to any scaling $$\fy(x)\mapsto e^{a\s}\fy(e^{b\s}x)$$ except for $S_{-1}^\s$ gives a non-zero constant multiple of $K(\fy)$. This is a special feature of the scaling critical case, which allows us to work with a single $K$, whereas in the subcritical case \cite{NakS} we needed two different functionals and their equivalence.

The main result of this paper is summarized as follows.

\begin{thm}\label{thm: Main}
There exist a small $\e>0$, a neighborhood $\B$ of $\pm\Sm_0$ within $O(\e)$ distance in $\HH$, and a continuous functional $\Sg:\HH^\e\setminus\B\to\{\pm 1\}$, where
\EQ{
 \HH^\e:=\{\U\fy\in\HH \mid E(\U\fy)<J(W)+\e^2\},}
such that the following properties hold: For any solution $u$ in $\HH^\e$ on the maximal existence interval $I(u)$, let
\EQ{
 \pt I_0(u):=\{t\in I(u)\mid \V u(t)\in \B\},
 \pr I_\pm(u) := \{t\in I(u) \mid \V u(t)\not\in\B,\ \Sg(\V u(t))=\pm 1\}.}
Then $I_0(u)$ is an interval, $I_+(u)$ consists of at most two infinite intervals, and $I_-(u)$ consists of at most two finite intervals. $u(t)$ scatters to $0$ as $t\to\pm\I$ if and only\break if $\pm t\in I_+(u)$ for large $t>0$. Moreover, there is a uniform bound $M<\I$ such that
\EQ{
 \|u\|_{L^q_{t,x}(I_+(u)\times\R^d)}\le M, \pq q:=\frac{2(d+1)}{d-2}.}
For each $\s_1,\s_2\in\{\pm\}$, let $A_{\s_1,\s_2}$ be the collection of initial data $\V u(0)\in\HH^\e$, and for some $T_-<0<T_+$,
\EQ{
 (-\I,T_-)\cap I(u) \subset I_{\s_1}(u), \pq (T_+,\I)\cap I(u) \subset I_{\s_2}(u).}
Then each of the four sets $A_{\pm,\pm}$ has non-empty interior, exhibiting all possible combinations of scattering to zero/finite time blowup as $t\to\pm\I$, respectively.
\end{thm}

The radial version of this exact theorem was proved in~\cite{CNLW1}.
The main difference from that paper is of course the presence of the translation and Lorentz symmetries which need to be taken into account.
Actually, the Lorentz symmetry does not play much role under the energy constraint $E(\vec u)<J(W)+\e^2$, where the solution can approach to $\pm\Sm$ only if $|P(\vec u)|\lec\e$.
In contrast, the translational freedom is not a priori controlled by conserved quantities, and so we instead eliminate it by suitable orthogonality conditions. In other words, the modulation theory here amounts to a system of $d+1$ ODEs corresponding to the dilation and translation symmetries.

By using the Lorentz transform, we can extend the above result to bigger energy, depending on the size of momentum.

\begin{cor}\label{cor: Main}
There exist a small $\e>0$, a neighborhood $\B$ of $\pm\Sm$ in $\HH$, and a continuous functional $\Sg:\ti\HH^\e\setminus\ti\B\to\{\pm 1\}$, where
\EQ{
 \ti\HH^\e:=\Big\{\U\fy\in\HH \mid E(\U\fy)<\sqrt{J(W)^2+\e^4+|P(\U\fy)|^2}\Big\}, }
such that the same conclusion holds as in Theorem \ref{thm: Main} if we replace $\HH^\e$ and $\B$ by $\ti\HH^\e$ and $\ti\B$ respectively.
Moreover, if $|E(\vec u)|\le|P(\vec u)|$ and $u\not\equiv 0$ then the solution blows up both in $t<0$ and in $t>0$.
\end{cor}
Actually, we can reduce the corollary to the theorem only for $|E(\vec u)|>|P(\vec u)|$, where we can find a Lorentz transform from $u$ to another solution $w$ with $E(\vec w)=\sqrt{E(\vec u)^2-|P(\vec u)|^2}$ and $P(w)=0$, see \eqref{eq:EPtrans}.

The other case $|E(\vec u)|\le|P(\vec u)|$ is treated separately, which is essentially known.
Indeed, for such a solution $u$, there is a Lorentz transform to another solution $w$ with $E(w)<J(W)$.
If the original solution $u$ is global in one direction, then so is $w$ (see Lemma \ref{global Lorentz}).
Then we have $K(w(0))\ge 0$, otherwise the classical result of Payne-Sattinger \cite{PS} (or more precisely by Kenig-Merle \cite{KM2} in the current setting) implies that $w$ blows up in both directions.
Then
\EQ{
 0\le \|\dot w\|_2^2/2+\|\na w\|_2^2/d=E(\vec w)-K(w)/2^*\le E(\vec w).}
This is already a contradiction if $|E(\vec u)|<|P(\vec u)|$, since then we can make $E(w)<0$.
In the remaining case $|E(\vec u)|=|P(\vec u)|$, we can make $E(\vec w)$ as small as we wish.
Hence the above inequality implies that the energy norm can be made arbitrarily small.
Then the small data scattering implies that $\|w\|_{L^q_{t,x}(\R^{1+d})}\lec\sqrt{E(\vec w)}$.
However, since $\|u\|_{L^q_{t,x}(\R^{1+d})}$ is Lorentz invariant, this implies that the original solution $u\equiv 0$.

The rest of the paper is devoted to proving the main theorem. We differ strongly from~\cite{CNLW1} in terms of the basic formalism which defines our approach. To be more precise, we perform a change of coordinates
in the time variable which allows us to work with a fixed reference Hamiltonian in the perturbative analysis rather than a moving one as in~\cite{KM2}.
This leads to some simplifications in the ejection lemma, for example, see Lemma~\ref{lem: ejection}.
We remark that the formalism is also different from the one used in the nonradial subcritical equation~\cite{NakS2}, where a complex formulation was chosen, and more essentially, in the choice of orthogonality conditions, which also brings some simplification.

One application of Theorem~\ref{thm: Main}  is the following corollary, which removes the radial assumption from~\cite[Corollary~6.3]{DKM3}. The solution~$W_+$ is the one discovered by Duyckaerts, Merle~\cite{DM2}. It is a radial $\dot H^1\times L^2$ solution, exists globally in forward time and approaches $W$ in~$\dot H^1$, and blows up in finite negative time.  As above, the dimension satisfies  $d=3$ or $d=5$.

\begin{cor}\label{cor:DM}
 Let $u$ be an energy solution of \eqref{main PDE} such that $E(\vec u) = E(W,0)$.
 Denote by $(u_0,u_1)$ the initial conditions of~$u$.  Assume that $$\int |\nabla u_0|^2\, dx > \int |\nabla W|^2\,dx$$ Then $u$ blows up in finite time in both time directions or $u = W_+$ up to the symmetries of the equation.
\end{cor}

In~\cite[Theorem 2, (c)]{DM2}  this result is proved (nonradially) under the additional condition that $u_0\in L^2$. Using~\cite{CNLW1}, this $L^2$ condition was removed in~\cite{DKM3}, but only in the radial setting. As noted in~\cite[Remark 6.5]{DKM3}, the removal of the radial assumption in~\cite{CNLW1} would then complete~\cite{DM2}  in the sense that the $L^2$-condition can be removed even nonradially.
This is what we accomplish in this paper, whence Corollary~\ref{cor:DM}. For the proof, we refer the reader to~\cite{DKM3}.

\section{The basic setup}

\subsection{The critical wave equation, Hamiltonian formalism}
The Cauchy problem for $\vec u=(u,\dot u)$
\EQ{
 \vec u_t = J\D\vec u + (0,|u_1|^{p-1}u_1), \pq J=\mat{0 & 1 \\ -1 & 0}, \pq \D=\mat{-\Delta & 0 \\ 0 & 1} }
is locally wellposed in $\HH$, with the conservation of energy:
\EQ{
 E(\vec u):=\frac{1}{2}\|\vec u(t)\|_\HH^2 - \frac{1}{2^*}\|u(t)\|_{2^*}^{2^*}.}
 The equation is the Hamiltonian flow in $\HH$ with conserved Hamiltonian $E$ relative to the symplectic form
\EQ{
 \om(\uu,\uv):=\LR{J\uu,\uv}_{L^2}=\LR{u_2|v_1}-\LR{u_1|v_2}.}

\subsection{The translation and scaling symmetries}

Another feature of equation~\eqref{main PDE} is its invariance with respect to the scaling:
\EQ{
 \vec S^\s := S^\s_{-1}\otimes S^\s_0, \pq S^\s_a\fy(x):=e^{(d/2+a)\s}\fy(e^\s x),}
which is a unitary group acting on $\HH$, with the generator
\EQ{
\vec \La \,  :=  \vec S'=\La_{-1}\otimes\La_0, \pq \La_a:=r\p_r+d/2+a=S_a'.}
With $\La_a^*$ denoting the adjoint relative to $L^2(\R^d)$ one has $\La_a^*=-\La_{-a}$ and thus
  $\vec\La^*=-\La_1\otimes \La_0$.
Similarly, the unitary group of translations is denoted by
\EQ{
 (T^cv)(x):=v(x-c), \pq T'=-\na,\pq c\in\R^d.}
Our analysis in this paper is around the static Aubin solution
\EQ{
 W(x) = \Big(1+\frac{|x|^2}{d(d-2)}\Big)^{1-\frac{d}{2}}, \pq -\De W=W^p,}
whose vector and scaled versions are denoted by
\EQ{
 \vec W:=(W,0), \pq \vec W_\s:=\vec S^\s\vec W=:(W_\s,0), \pq \vec \La \,  \vec W=:(W',0).}
Let $\vec u=\uu=T^c \vec S^\s(\vec W+\uv)$ be a solution with $\s=\s(t)$ and $c=c(t)$.
In general, $\uv$ need not have the structure~\eqref{eq:uvec}, which is why we do not write $\vec v$. Noting that
\EQ{
 \pt \na S_a^\s = S_a^\s e^\s\na, \pq \na S_a'=(S_a'+1)\na,
 \pr (T^c\uu)_t = T^c(\uu_t-\dot c\na \uu), \pq
 (\vec S^\s \uu)_t=\vec S^\s(\uu_t+\dot\s\vLa \uu),}
 where $\dot c=c_t$, $\dot\s = \s_t$,
we obtain the equation of $\uv$:
\EQ{ \label{eq v}
 \uv_t = e^\s[J\LL \uv + \uN(\uv)] +  ( e^\s \dot c\cdot\na - \dot\s\vec \La)(\vec W+\uv)}
where the linearized and superlinear operators are defined by
\EQ{
 \pt \LL=\mat{L_+ & 0 \\ 0 & 1}, \pq L_+=-\De-pW^{p-1},
 \pr \uN(\uv)=(0,N(v_1)), \pq N(v)=(W+v)^p-W^p-pW^{p-1}v.}
 The structure of the spectrum of $L_+$ over $L^2(\R^d)$ is as follows:
   the discrete spectrum consists of a unique negative
 eigenvalue of $L_+$ which we denote by~$-k^2$. The associated eigenfunction is the ground state of~$L_+$, denoted by $\rho$:
\EQ{
 L_+\rh=-k^2\rh, \pq \rh>0, \pq \|\rh\|_2=1.}
The essential spectrum of $L_+$ is $[0,\I)$, and it is purely absolutely continuous. At the threshold $0$,
one has an eigenvalue of multiplicity~$d$, with eigenfunctions $\na W$, and a resonance function (or an eigenfunction if $d\ge 5$) $W'=\La_{-1}W$
which is unique.

 \subsection{A change of time and the static linearized operator}

The time-depende-\\ nt coefficient on the linearized operator is removed by the standard change of time variable from $t$ to $\ta$:
\EQ{ \label{eq vta}
 \pt \frac{d\ta}{dt}=e^{\s(t)},
 \pq \uv_\ta = J\LL \uv + \uN(\uv) +(e^\s  c_\ta  \cdot\na - \s_\ta \vec \La \, )(\vec W+\uv).}
The ``generalized" eigenvectors of $J\LL$ are
\EQ{\label{modes}
 \pt J\LL \na\vec W=0, \pq J\LL J\na\vec W=-\na\vec W, \pq J\LL \vec \La \,  \vec W=0, \pq J\LL J\vec \La \,  \vec W=-\vec \La \,  \vec W,
 \pr J\LL \ug^\pm = \pm k \ug^\pm, \pq \ug^\pm:=(1,\pm k)\rh/\sqrt{2k},}
where $\rh$ is the aforementioned ground state of $L_+$. The normalization here is such that $\om(\ug^+, \ug^-)=1$.  Define
$\vec\rho:= (\rho,0) $.
Note that $\vec \La \,  \vec W\not\in L^2$ for $d<5$. Hence we decompose
\EQ{\label{v exp}
 \pt \uv= \la_+\, \ug^+ + \la_-\, \ug^- + \mu\cdot\na\vec W + \uga
 \pr \la_\pm:= \om(\uv,\pm \ug^\mp), \pq \mu:=\om(\uv,J\na\vec\rho)/a_W=\LR{ v_{1}  | \na \rho     }/a_{W}, }
where
\EQ{
 a_W := \frac{1}{d}\LR{-\Delta W|\rho} = \frac{1}{d} \LR{ W^{2^{*}-1}| \rho}. }
By construction,
\EQ{
\om(\uga, \ug^\mp)=0, \qquad \om(\uga, J\nabla \vec \rho)=0
}
The more natural $ \mu:=\om(\uv,J\na\vec W)=\LR{v_{1}|\na W}$ is problematic since the latter inner product is
not well-defined.

Note that we did not extract the remaining root-mode $J\na\vec W$ from $\uga$, which corresponds to Lorentz ``boosts'', i.e., translations in momentum.

\subsection{Energy expansion}
Using \eqref{v exp} the energy is expanded as
\EQ{
 E(\vec u)-E(\vec W) = \frac{1}{2}\LR{\LL \uv|\uv}-C(\uv) = -k\la_+\la_- + \frac12\LR{\LL\uga|\uga} - C(\uv),}
where the superquadratic part is given by
\EQ{
 C(\uv):=\int_{\R^d}  \Big[ \frac{|W+v_1|^{2^*}-W^{2^*}}{2^*}-W^pv_1-\frac{p}{2}W^{p-1}|v_1|^2 \Big](x)\, dx.}
 One has the estimate
 \EQ{
\label{C est}
|C(\uv)|\lec \| W^{2^{*}-3} v_{1}^{3}\|_{1} + \|v_{1}\|_{2^{*}}^{2^{*}}
 }
 since $2^*>3$.

 \begin{lem}\label{lem:L+ spectral}
 Let $f\in \dot H^1(\R^d)$ with $f\perp \rho$. Then
 \EQ{\label{L+ pos}
 \LR{ L_+ f | f}\ge 0\\
 }
 and
 \EQ{\label{norm comp}
  \LR{ L_+f| f} + |\LR{ f|\Lambda_0 \rho}|^2 + |\LR{ f|\na \rho}|^2 \simeq \|\na f\|_2^2
 }
 where the implicit constants in~\eqref{norm comp} only depend on the dimension.
 \end{lem}
 \begin{proof}
 The first statement follows from the self-adjointness of $L_+$ and the description of the spectrum of~$L_+$.
 For the second, we need to invoke the calculus of variations.
 It is clear from Sobolev imbedding and H\"older's inequality that the right-hand side of~\eqref{norm comp} dominates
 the left-hand side.
 Suppose the reverse inequality of~\eqref{norm comp} fails.
 Then there exists a sequence $\{f_n\}\subset \dot H^1(\R^d)$ with $\|\na f_n\|_2=1$
 and $f_n\perp\rho$ which further satisfies
 \EQ{\label{drei punkt}
 \LR{ L_+ f_n|f_n} \to 0,\quad
 \LR{ f_n|\La_0\rho} \to 0,\quad
 \LR{ f_n | \na \rho} \to0
 }
 After passing to a subsequence we may assume that $f_n\weakto f_\I$ in $\dot H^1(\R^d)$ and $L^{2^*}(\R^d)$, as well
 as $f_n\to f_\I$ strongly in $L^2_{\mathrm{loc}}(\R^d)$.
 Then $f_\I\perp\rho$,  and
 \EQ{ \label{lin constraint}
 \LR{f_\I | \Lambda_0\rho} =0,\quad \LR{ f_\I | \na \rho}=0
 }
 From the local convergence in $L^2$ we conclude that
 \EQ{\label{peinlich}
 \int_{\R^d} W^{2^*-2}(x) |f_n(x)|^2\, dx \to  \int_{\R^d} W^{2^*-2}(x) |f_\I(x)|^2\, dx
 }
 By the first condition in~\eqref{drei punkt} the left-hand side in~\eqref{peinlich} also tends to~$1/p$.
 But then $\LR{L_+ f_\I | f_\I} \le 0$, and from~\eqref{L+ pos} we conclude that $\LR{L_+ f_\I | f_\I} = 0$
 whence
 \[
 \|\na f_n\|_2 \to \|\na f_\I \|_2  \text{\ \ as\ \ }n\to\I
 \]
 Finally, this means that $f_n\to f_\I$ strongly in $(\dot H^1\cap L^{2^*})(\R^d)$. In summary, $L_+f_\I=0$ and so
\[
f_\I = \alpha W' + \vec \beta \cdot \nabla W
\]
Inserting this into \eqref{lin constraint} implies that $\al=0$ and $\vec \beta=0$. To see this,
we first note that $\LR{  W' |\La_0\rh}\ne0$ which follows from
$\LR{ W' | \rho}=0$ and
\EQ{\label{pos quant}
b_W:=\LR{  W' |\La_0\rh} &= -k^{-2}\LR{[L_+,x\cdot\nabla]   W' |\rh} \\
&=  k^{-2}  \LR{ (2L_+ + p(p-1) W^{p-2}W' )   W' |\rh} \\
&= k^{-2} p(p-1)  \LR{W^{p-2}(W')^2|\rh}>0
}
Since $\LR{\nabla W|\La_0\rh}=0$, we infer from this that $\al=0$.
On the other hand,
\[
\LR{ \na_j W | \na_k \rho} = -\frac{1}{d} \de_{jk} \LR{ -\Delta W|\rho}, \pq \LR{-\De W|\rh} =  \LR{ W^{2^*-1} | \rho}>0
\]
and so $\vec \beta=0$.
But this clearly contradicts $\|\na f_\I \|_2=1$, whence we have arrived at a contradiction.
\end{proof}

In what follows we denote
\EQ{
 \al:=\LR{v_1|\La_0\rh}=\LR{\ga_1|\La_0\rh}}
 so that \eqref{norm comp} implies the following: for any $\uga$ with $\ga_1\perp \rho, \na\rho$ we have
 \EQ{
 \|\uga\|_\HH^2 \simeq \LR{\LL\uga|\uga} +\alpha^2
 }

\subsection{Orthogonality conditions near the ground state}
Now we introduce the crucial {\em orthogonality conditions} near the family of static solutions $\pm\Sm_0$.
Indeed, we claim that for any $\uu\in\HH$ with
\EQ{\label{S0 close}
\min_{\pm}\dist_{\HH}(\uu, \pm\Sm_0) \ll 1 }
admits the representation $\uu=T^c\vec S^\s(\pm\vec W+\uv)$ where we can choose $(\s,c)\in\R^{1+d}$ such that
\EQ{\label{orth}
 \pt 0=\al=\LR{v_1|\La_0\rh},
 \pq 0=\mu=\LR{v_1|\na \rho}=\om(\uv,J\na\vec \rho).}
\eqref{orth} are the orthogonality conditions which we use in this paper.
To verify this claim, take any $v_1$ small in $\dot H^1$ (and thus small in $L^{2^*}(\R^d)$), and define (taking $+\Sm_0$ for simplicity)
\EQ{
 F(\sigma,c):= \LR{ S_{-1}^{-\s}T^{-c}(W+v_1)-W|(\Lambda_0\rho,\na\rho)} \in \R^{1+n}. }
Note that $F(0)=\LR{v_1|(\Lambda_0\rho,\na\rho)}$ is small and $$F'(0)=\LR{W+v_1|(\La_1,\na)\otimes(\Lambda_0,\na)\rho}$$ is close to $$\LR{W|(\La_{1},\na)\otimes(\Lambda_0,\na)\rho}=-\diag(b_W,a_W,\dots,a_W).$$
It follows from the inverse function theorem that there exists $(\s,c)$ small
with $F(\s,c)=0$. But then $\tilde v_1$ defined by means of
\EQ{
 W+v_1=S_{-1}^\s T^{-c}(W+\tilde v_1)}
satisfies $\LR{\tilde v_1|\Lambda_0\rho}=0$ and $\LR{\ti v_1|\na\rho}=0$.

The orthogonality conditions~\eqref{orth}  are equivalent to $\alpha=0$ and $\mu=0$ in~\eqref{v exp},
which ``eliminate'' the dilation and translation symmetries, respectively.

\subsection{Linearized energy}
Change variables from $(\la_+, \la_-)$ to $(\la_1,\la_2)$ as follows:
\EQ{ \label{la change}
 \la_1=\frac{1}{\sqrt{2k}}(\la_++\la_-), \pq \la_2=\sqrt{\frac{k}{2}}(\la_+-\la_-),\pq \la_\pm = \sqrt{\frac{k}{2}}\, (\la_1\pm k^{-1}\la_2).
}
Note that in these variables we have $\la_j=\LR{v_j|\rho}$ ($j=1,2$) and
\EQ{
 v_1 = \lambda_1 \rho + \mu\nabla W +\gamma_1, \pq v_2= \la_2 \rho + \ga_2.
}
Now we define the nonlinear energy distance near $\pm\Sm_0$ by means of the equations
\EQ{
 \E(\vec u) &:=E(\vec u)-J(W)+k^2\la_1^2+\al^2 + |\mu|^2   \\
 &= \frac12(k^2\la_1^2+\la_2^2)+\frac12\LR{\LL\uga|\uga}+\al^2 + |\mu|^2  -C(\uv).}
Here $\uu=T^cS^\s(\pm\vec W+\uv)$ with  $\|\uv\|_{\HH}$ small and some choice of $\pm$, and
 we use the decomposition~\eqref{v exp}.
Lemma~\ref{lem:L+ spectral}  together with $C(\uv)=o(\|\uv\|_\HH^2)$, implies that
\EQ{
 \E(\vec u)  \simeq \|\uv\|_\HH^2.}
Hence it is natural to define the linearized energy norm by
\EQ{\label{E def}
 \|\uv\|_E^2 := \frac12(k^2\la_1^2+\la_2^2) + \frac12\LR{\LL\uga|\uga} + \al^2 + |\mu|^2.  }

\subsection{Modulation equations}
Differentiation of the orthogonality conditions \eqref{orth} in $\ta$ using the equation \eqref{eq vta} yields
\EQ{ \label{modulat eq}
 \pt 0=\p_\ta\LR{v_1|\La_0 \rh} = \LR{v_2|\La_0\rh} - c_\ta\, e^\s\LR{v_1|\na\La_0 \rh} - \s_\ta [b_W-\LR{v_1|\La_{1}\La_0\rh}],
 \pr 0=\p_\ta\LR{v_1|\na \rho} = \LR{v_2|\na \rho} +  c_\ta\, e^\s[a_W I_d -\LR{v_1|\na^2 \rho}] + \s_\ta \LR{v_1|\La_1\na \rho} ,}
where $a_W,b_W >0$ are as in~\eqref{v exp} and \eqref{pos quant}, $I_d$ is the $d$-dimensional unit matrix, and $\na^2\rho$ is the Hessian. The modulation equations \eqref{modulat eq} determine the evolution of $(\s,c)$ as long as $\uv$ remains small in $\HH$.
 For future reference, we remark that  in the notation of~\eqref{uv decomp}
\EQ{
 \LR{v_{2}|\La_{0}\rho} = \LR{\ga_{2} | \La_{0}\rho},\qquad \LR{v_{2}|\na\rho} = \LR{\ga_{2}|\na\rho}}
whence
\EQ{ \label{modul est}
  |\s_\ta| +  |c_\ta| e^{\s} &\lec \|\uga\|_{\HH},
}
as long as $\|\uv\|_E$ is small.

\subsection{Hyperbolic drivers}
On the other hand, the unstable/stable modes evolve by the following equations derived from \eqref{eq vta} with $\la_\pm=\om(\uv,\pm\ug^\mp)$
\EQ{\label{la+ la-}
 \pt \p_\ta\la_\pm = \pm k\la_\pm \mp c_\ta \, e^\s \om(\uv, \na \ug^{\mp}) \mp \s_\ta  \, \om(\vec\La \uv, \ug^{\mp}) \pm \om(\uN(\uv), \ug^{\mp}).}
In terms of $\la_1$ and $\la_2$ these equations become
\EQ{\label{hyp driver}
\pt \p_\ta \la_1 = \la_2 - c_\ta \, e^\s \, \LR{v_1|\na \rho} + \s_\ta  \LR{v_1|\La_1\, \rho}
\pr \p_\ta \la_2 = k^2\la_1 -   \,  c_\ta \, e^\s \, \LR{ v_2|\na\rho} +   \,  \s_\ta  \LR{  v_2| \La_0 \rho} +   \LR{ N(v_1)| \rh}.
}
The relation between these systems is given by \eqref{la change}.
Under the orthogonality conditions~\eqref{orth} the first equation in~\eqref{hyp driver} simplifies to
\EQ{
 \p_\ta \lam_{1}=\lam_{2} +\s_\ta\la_1.}
This will be important to guarantee convexity of the distance function in the ejection lemma. However it should not be confused with the definition $\partial_t u_1=u_2$, even though $\la_j=\LR{v_j,\rho}$, as $t$ and $\ta$ are different.

\section{Distance function, $\lambda$ dominance, ejection}

The following lemma establishes the existence of a distance function associated with the soliton manifold $\pm\Sm_0$ in such a way that near this manifold the distance function is proportional to the unstable mode in a suitable sense.

\begin{lem}
\label{lem:lambda dom}
There exists $\de_E>0$ and $d_W(\uu):\HH\to[0,\I)$ continuous such that
\EQ{ \label{energy dist}
 \pt d_W( \uu )\simeq \inf_{\pm, \s, c} \| \uu\mp \vec W_\s(\cdot-c)\|_{\HH},
 }
 and so that
for any $d_W( \uu)\le\de_E$  there exists a unique vector $(\sg,\s,c)\in \{\pm1\}\times \R^{1+d}$
with
\EQ{ \label{orth decop}
\uu = T^{c}\vec S^{\s}(\sg\vec W+\uv), \quad \LR{v_1|\La_0\rh}=0,\;\; \LR{v_1|\na\rh}=0.}
Decomposing
\EQ{\label{uv decomp}
  \uv = \la_+\ug^{+} + \la_-\ug^{-}+\uga,
 \pq \la_\pm=\om( \uv,\pm \ug^\mp),}
we have
\EQ{
 d^{2}_W( \uu) \simeq \| \uv\|_E^2 = k^2(\la_+^2+\la_-^2)+\frac 12 \LR{\LL\uga|\uga}}
In addition, if
\EQ{ \label{outer region}
 2(E( \uu)-J(W))<d^{2}_W( \uu)<\de_E^2}
then  $d_W(\uu)\simeq|\la_1|=|\la_++\la_-|$. Finally,
\EQ{\label{dWu small}
d_W( \uu ) \le \de_E \implies d^{2}_W( \uu)=E( \uu)-J(W)+ k^2\la_1^2
}
\end{lem}
\begin{proof}
There are $0<\de_A\ll 1$ and $C \ge 1$ such that putting $d_0(\uu):=C\dist_\HH(\pm\Sm_0,\uu)$, the above arguments starting from \eqref{S0 close} work in the region $d_0(\uu)\le\de_A$, and
\EQ{
 d_0(\uu)^2\simeq E(\uu)-J(W)+k^2\la_1^2=:d_1^2(\uu) \le d_0^2(\uu).}
Hence there exists $\de_E\in(0,\de_A)$ such that
\EQ{
 d_0(\uu)\le\de_A \text{ and } d_1(\uu)\le\de_E \implies d_0(\uu)\le \de_A/2.}
Now we choose a smooth cutoff function $\chi\in C^\I(\R)$ satisfying $\chi(r)=1$ for $|r|\le 1$ and $\chi(r)=0$ for $|r|\ge 2$ and set
\EQ{
 d_W( \uu) := \chi(2d_0(\uu)/\de_A)d_1(\uu) + (1-\chi)(2d_0(\uu)/\de_A)d_0(\uu). }
If $d_0(\uu)\ge\de_A$ then $d_W(\uu)=d_0(\uu)\ge\de_A>\de_E$. Hence if $d_W(\uu)\le\de_E$ then $d_0(\uu)<\de_A$ and $d_1(\uu)\le d_W(\uu)\le\de_E$, so $d_0(\uu)\le\de_A/2$, $d_W(\uu)=d_1(\uu)$.
The stated properties now follow easily from the considerations in the previous section.
\end{proof}

The following lemma is the analogue of the ``ejection lemma'' in our previous papers, see~\cite{NakS}.
As usual, we shall need the Payne-Sattinger functional
\EQ{
K(u) =  \int_{\R^{d}} \big[  |\nabla u|^{2} -  |u|^{2^{*}} \big]\, dx
}
in our analysis of the global dynamics, which is why it appears below.

\begin{lem}\label{lem: ejection}
There exists $\de_H\in(0,\de_E)$ with the following properties:
Let $u$ be a solution on an open interval $I$ such that for some $t_0\in I$
\EQ{ \label{away from S}
 \de_0:=d_W(\vec u(t_0))\le \de_H, \pq E(\vec u)-J(W) \le \de_0^2/2,}
and
\EQ{ \label{outgoing}
  \p_t d_W(\vec u(t_0))\ge 0.}
Apply the decomposition from Lemma~\ref{lem:lambda dom}.
Then for $t>t_0$ in $I$ and as long as $d_W(\vec u(t))\le\de_H$, $d_W(\vec u(t))$ is increasing, and
\EQ{ \label{eject est}
 \pt d_W(\vec u(t)) \simeq -\sg\lambda_+(t)\simeq -\sg\lam_1(t) \simeq e^{k\tau }\de_0,}
where $\ta(t)$ is the solution of the ODE
$\ta'(t)=e^{\s(t)}$ with $\ta(t_0)=0$.
Moreover,
\EQ{
 \pt |\s(t)-\s(t_0)|\lec d_W(\vec u(t)),
 \pr \sg K(u(t)) \gec   d_W(\vec u(t)) - C_*\, d_W(\vec u(t_0))  ,
 \pr  |\lam_-(t)|+\| \uga(t)\|_{\HH} \lec \de_0 + d_W^2(\vec u(t)),}
for some absolute constant $C_*>0$ and $\sg=\pm 1$ is fixed on the time interval.
\end{lem}
\begin{proof}
By Lemma~\ref{lem:lambda dom} and \eqref{away from S}, we conclude that $|\lambda_1(t_0)|\simeq \de_0$.
Furthermore, as long as $d_W(\vec u(t))$ remains sufficiently small and one has $d_W(\vec u(t))\ge \de_0$, the relation
\EQ{\label{rel}
|\lam_1(t)|\simeq d_W(\vec u(t))} is preserved.
In particular, if $d_W(\vec u(t))$ is increasing this relation is preserved. We shall therefore {\em assume} \eqref{rel} in our argument that establishes the monotonicity and~\eqref{eject est}.
The logic here is that once we have shown these properties to be correct, then the validity of~\eqref{rel} follows a posteriori by the method of continuity.

Differentiating \eqref{dWu small} using \eqref{hyp driver} and \eqref{modulat eq} as well as \eqref{orth decop} yields,
\EQ{\label{d d2}
 \pt \p_\tau d^2_W(\vec u) = 2k^2\lam_1\p_\ta\lam_1 = 2k^2\la_1(\la_2+\s_\ta\la_1)}
and
\EQ{\label{convex}
 \p_\tau^2 d^2_W(\vec u) = 2k^2(k^2\lam_1^2+\lam_2^2) + O(\lam_1^3).
}
In conjunction with the previous lemma we conclude from \eqref{convex} that $d_W^2(u)$ is increasing and convex in $\ta$,
as long as it remains sufficiently small.
Next, we remark that $\lam_1$ and $\p_\ta \lam_1=\lam_2+\s_\ta\la_1$ have the same sign, since
\EQ{
 2k^2\la_1(\la_2+\s_\ta\la_1)=\p_\ta d_W^2(\vec u) \ge \p_\ta d_W^2(\vec u)|_{t=t_0} = 2d_W(\vec u(t_0))\p_\ta d_W(\vec u(t_0))\ge 0.}
This implies that $|\lam_+|\gec|\lam_-|$ and thus $|\lam_1|\simeq |\lam_+|$.

The evolution of $\lam_+$ in $\ta$ is determined by~\eqref{la+ la-}, which states that
\EQ{
 \p_\ta\lam_\pm = \pm k\lam_+ + O(d_W^2(\vec u) ). }
Since $d_W(\vec u)\lec |\lam_+|$, we see that $|\lam_+|\simeq e^{k\ta}\de_0$, and
\EQ{
 |\lam_-|\lec \de_0+\lam^2_+}
as claimed. As for the scaling parameters, \eqref{modulat eq} yields $|\p_\ta\s|\lec\|\ga\|_\HH\lec d_W(\vec u)$, and hence integrating the exponential bound in $\ta$ implies $|\s-\s(t_0)|\lec d_W(\vec u)$.

The $\uga$-part is estimated as in Lemma~4.3 of~\cite{NakS}. To this end define
\EQ{
 \uv_{d}:=\lam_{+}\ug^{+}+\lam_{-}\ug^{-}=\uv-\uga.}
Then
\EQ{
E(\vec W + \uv) &= J(W) + \frac12(\lam_{2}^{2}- k^2 \lam_{1}^{2}) + \frac12 \LR{\cL \uga|\uga}- C(\uv) \\
E(\vec W + \uv_{d}) &= J(W) + \frac12(\lam_{2}^{2}- k^2 \lam_{1}^{2}) - C(\uv_{d})
}
whence
\EQ{
E(\vec u) - E(\vec W+\uv_{d}) &= \frac12\LR{ \cL\uga|\uga} + C(\uv_{d})- C(\uv)
}
as well as (where $\uv_{d,1}$ denotes the first component of $\uv_{d}$)
\EQ{
 \p_\ta E(\vec W+ \uv_{d}) &=\lam_{2}\p_\ta \lam_{2} - k^2\lam_{1}\p_\ta\lam_{1} - \LR{ N(\uv_{d,1}) | \partial_{\ta} \uv_{d,1}}  \\
&=     \lam_{2}\big(\LR{N(v_{1})|\rho} + \s_\ta \LR{v_{2}|\La_{0}\rho} - c_\ta e^{\s} \LR{v_{2}|\nabla \rho}    -  \LR{ N(\uv_{d,1}) |  \rho} \big ) -k^2\s_\ta\la_1^2
 \pr= O\big(  d_W^2(\vec u) \, \|\uga \|_{\HH} \big)
}
where we used \eqref{modul est} to bound $\s_\ta$ and $c_\ta e^{\s}$.
In view of the preceding,
\EQ{\nn
|E(\vec u)- E(\vec W + \uv_{d} ) | &\lec |E(\uu(t_0)) - E(\vec W + \uv_{d} (t_0)) | + \\
&\qquad + |E(\vec W + \uv_{d} ) - E(\vec W + \uv_{d} (t_0)) |  \\
&\lec \de_0^2 + \|\uga\|_{L^\I([t_0,t], \HH)} \, d_W^2(\vec u)
}
where we used the exponential growth of $d_W(\vec u)$ to pass to the last line.  Furthermore,
\EQ{
\frac12 \LR{\cL\uga|\uga} &\le  |E(\vec u)- E(\vec W + \uv_{d} ) | + | C(\uv_{d})- C(\uv)  | \\
&\lec \de_0^2 +  \|\uga\|_{L^\I([t_0,t], \HH)} \, d_W^2(\vec u)
}
In conclusion,
\EQ{
\| \uga\|_{\HH} \lec \de_0 +  d_W^2(\vec u)
}
as claimed.

Finally, expanding the $K$-functional, one checks that
\EQ{\label{K W v}
 K(W+v_1) = -(2^*-2)\LR{W^{2^*-1}|v_1} + O(\|v_1\|_{\dot H^1}^2)
}
Inserting the expansion $v_1= \la_1 \rho + \ga_1$  into \eqref{K W v} and
using the bounds on $\la_1(t)$ and $\uga(t)$ that we just established implies
the desired properties of $K$.
\end{proof}

We remark that unlike the subcritical nonradial paper~\cite{NakS2} the distance function is convex near a minimum and thus increasing in Lemma~\ref{lem: ejection}.
The difference lies with the choice of orthogonality conditions corresponding to the translational symmetry, which in our case insure that $\p_\ta\lam_1=\lam_2+\s_\ta\la_1$.
This is similar to the behavior in the radial subcritical case, see~\cite{NakS}.

 \section{The variational structure in the energy critical setting}

We recall the following characterization of the ground state:
 \EQ{
 J(W)&=\inf\{ J(\fy)\mid K(\fy)=0, \fy\in \dot H^1(\R^d),\fy\ne0\} \\
 &= \inf\{ \frac{1}{d} \|\nabla \fy\|_2^2 \mid K(\fy)\le0, \fy\in \dot H^1(\R^d),\fy\ne0\}
 }
where $J(\fy)$ is {\em the static energy} defined in \eqref{def J},
and $\pm W$ are the unique minimizer up to the dilation (as in $W_\s$) and translation symmetries. In other words, $W$ is the unique (up to the same symmetries) extremizer of the Sobolev embedding $\dot H^1(\R^d)\embed L^{2^*}(\R^d)$.
 We need the following variational structure outside of the soliton tube.

\begin{lem}\label{lem: variational} For each $0<\delta<1$ there is
 $\eps_1 = \eps_1(\de)>0$   such that if $\uu\in\HH$ satisfies $J(u_1)<J(W)+\eps_1^2(\de)$ and ${d}_W(\uu)>\delta$, then we have either
\EQ{
 K(u_1) >\min\{\kappa(\de),\,c\|\nabla u_1\|_{L^2}^2\}}
 or else
\EQ{
 K(u_1) <-\kappa(\de)}
 for suitable $\kappa(\de)>0$ and some absolute constant $c>0$.
 \end{lem}

\begin{proof}
We first eliminate the $u_2$ component from $\uu$: if $\|u_2\|_2\ll \de$, then it follows that $d_W(u_1,0)>\de/2$. On the other hand,
if $\| u_2\|\simeq \de$, then assuming $\eps_1(\delta)\ll \de$ as we may, it follows that
\[
J(u_1)<J(W)-c\de^2
\]
with some absolute constant $c$. But then we must have $\| u_1 - W_\s(\cdot-c)\|_{\dot H^1}\gtrsim \de$ for all $\s,c$. Hence
\EQ{
 \de\lec d_W(u_1,0) \simeq \dist(u_1,\Sm_0)}
in all cases. In the rest of proof we regard $u=u_1\in\dot H^1(\R^d)$ with $\dist(u,\Sm_0)\gec\de$.

By the critical Sobolev imbedding, the statement holds provided $\|\nabla u\|_{2}< c_{0}$ where $c_{0}>0$ is some absolute constant.
Thus, assume the lemma fails and let $\{u_{n}\}_{n=1}^{\I}\subset \dot H^{1}$ be a sequence with
\EQ{\label{eq:seqdef}
\|\nabla u_{n}\|_{2}\to c\ge c_{0},\quad K(u_{n})\to0,\quad J(u_{n})< J(W)+\frac{1}{n}
}
as well as $\dist(u_n,\Sm_0)\gec \delta_0$.
Since
\EQ{\label{eq:Jun}
J(u_{n}) = \frac{1}{d} \| \nabla u_{n}\|_{2}^{2} + \frac{1}{2^{*}} K(u_{n})
}
we see that $\{u_{n}\}_{n=1}^{\I}$ is bounded in $\dot H^{1}\cap L^{2^{*}}$ and so $c<\I$.
Then the latter two conditions of \eqref{eq:seqdef} implies that $\{u_n\}$ is an extremizing sequence for the critical Sobolev embedding $\dot H^1(\R^d)\subset L^{2^*}(\R^d)$, and so, by the celebrated theorem of P.-L.~Lions \cite[Theorem I.1]{Lions}, it is compact in $\dot H^1$ up to scaling and translation, hence converging strongly to the unique minimizer $W$ up to scaling and translation.
But this clearly contradicts $\dist(u_n,\Sm_0)\gec \delta_0$.
\end{proof}

As in the previous works \cite{NakS}, we can define a sign functional by combining the ejection lemma with the  variational structure exhibited in the previous lemma.

\begin{cor} \label{cor:sign}
Let $\de_S:=\de_H/(2C_*)>0$ where $\de_H>0$ and $C_*\ge 1$ are the constants from~Lemma \ref{lem: ejection}. Let $0<\de\le\de_S$ and
\EQ{
 \HL_{(\de)}:=\{ \uu\in\HH \mid E( \uu)<J(Q)+\min(d_W^2( \uu)/2,\e_1^2(\de))\},}
where $\e_1(\de)$ is defined in Lemma~\ref{lem: variational}. Then there exists a unique continuous function $\Sg:\HL_{(\de)}\to\{\pm 1\}$ satisfying
\EQ{
 \CAS{  \uu\in \HL_{(\de)},\ d_W( \uu)\le\de_E &\implies \Sg( \uu)=-\sign\la_1,\\
  u\in\HL_{(\de)},\ d_W( \uu)\ge\de &\implies \Sg( \uu)=\sign K(u) ,}}
where we set $\sign 0=+1$.
\end{cor}
\begin{proof}
The proof is the same as in the subcritical radial case, see~\cite{NakS0}.
\end{proof}

\section{   The one-pass theorem           }
 A key step in the proof of our main theorem is to show that the sign $\Sg( \uu(t))$ can change at most once for any solution of~\eqref{main PDE}.
This goes by the name of {\em one-pass theorem}, see~\cite{NakS}. The current section is entirely devoted to this theorem:
\begin{thm}\label{thm: onepass}
There exist
 $0<\e_*\ll \de_*\ll\de_H$ with the following properties:  Let $\vec u\in C(I;\HH)$ be a solution of \eqref{main PDE} on an open interval $I$, satisfying for some $\e\in(0,\e_*]$, $\de\in(\sqrt 2\,\e,\de_*]$ and $T_1<T_2\in I$
\EQ{
 E(\vec u)\le J(W)+\e^2, \pq d_W (\vec u(T_1))<\de=d_W (\vec u(T_2)).}
Then   $d_W (\vec u(t))>\de$ for all $t>T_2$ in $I$.
\end{thm}
\begin{proof}
By increasing $T_1$ and decreasing $T_2$ if necessary, we may assume in addition that $\sqrt 2\,\e<d_W (\vec u(T_1))$ and $\p_t d_W (\vec u(t))|_{t=T_1}\ge0$. Then Lemma \ref{lem: ejection} applies for all $t\in[T_1,T_2]$ and so $d_W (\vec u(t))$ is increasing for $t>T_1$ until it reaches $\de_H$ (the small absolute scale in the ejection lemma). Arguing by contradiction, we assume that for some $t>T_2$ we have $d_W (\vec u(t))\le\de$. Such a $t$ can occur only away from $T_2$ (this will be made more precise shortly), and after $d_W (\vec u(t))$ has increased to size $\de_H\gg\de$. Moreover, by applying Lemma~\ref{lem: ejection} backward in time, we can find $T_3>T_2$ such that $d_W (\vec u(t))$ decreases from $\de_H$ down to $\de$ as $t\nearrow T_3$, and so
 that $$d_W (\vec u(t))>\de=d_W (\vec u(T_3)) =d_W (\vec u(T_2))$$ for $T_2<t<T_3$. We may further assume
\EQ{\label{fav dil}
 \s(u(T_2))=0 \le \s(u(T_3)),}
by rescaling and reversing time, if necessary.   Here $\s$ is defined in Lemma \ref{lem:lambda dom}.

We now proceed by combining the proof ideas of the analogous theorem for the critical radial wave equation~\cite{CNLW1} with that for the subcritical nonradial  Klein-Gordon equation, see~\cite{NakS2} (with slight improvement).
Following the latter reference, we first show that the centers of the ground state as given by the path $c(t)$, diverge from each other between times $T_2$ and~$T_3$ by an amount $\ll T_3-T_2$.
Once this is done, we shall adapt the virial argument from~\cite{CNLW1} to the nonradial context, which will then allow us to exclude almost homoclinic orbits. It will be understood that all times $t$ belong to the interval~$I$.

By spatial translation, we may assume that $c(T_2)=0$. By the ejection we have
\EQ{
 T_3-T_2 \gec \log(\de_E/\de) \gg 1,}
and by the finite speed of propagation
\EQ{\label{FSP}
 |c(T_3)| \le T_3-T_2+O(1),}
where $c(t)=c(u(t))\in\R^d$ is defined by Lemma \ref{lem:lambda dom} as long as $\vec u(t)$ is close to $\Sm_0$, which is true when $t$ is close to $T_2$ or $T_3$.
Consider a localized center of energy defined by (with $\vec u= (u_1,u_2)$)
\EQ{
 \CE(t):=\LR{xw|e(\vec u)},\pq e(\vec u):=[|u_2|^2+|\na u_1|^2]/2-|u_1|^{2^*}/{2^*},}
where $w(t,x)$ is the cut-off function onto a light cone defined by
\EQ{ \label{def w}
 w(t,x)=\chi(|x|/(t-T_2+S))}
for some $1\ll S=S(\de)<\delta^{-2}$ to be determined, and some $\chi\in C^\I(\R)$ satisfying $\chi(r)=1$ for $|r|\le 1.5$ and $\chi(r)=0$ for $|r|\ge 2$.
Then using the equation of $u$, we have
\EQ{
 \dot\CE(t)\pt=\LR{\dot wx|e(\vec u)}-P(\vec u)+\LR{(1-w)u_2|\na u_1}-\LR{x u_2|\na w\cdot\na u_1}
 \pr=O(\Ext(t))-P(\vec u),}
where $\Ext(t):=\|\vec u(t)\|_{\HH(|x|>t-T_2+S)}^2$ denotes the exterior free energy.
Hence,
 \EQ{
 |\CE(T_{3})-\CE(T_{2})|\lec (T_{3}-T_{2})\max_{T_{2}\le t\le T_{3}}(\Ext(t)+|P(\vec u)|). }
The conserved momentum is small because
\EQ{
 |P(\vec u)| \le |P(\vec u)-P(T^c\vec S^\s \vec W)| \lec \|\uv(T_1)\|_\HH+\|\uv(T_1)\|_\HH^2 \lec \de.}
Using the finite propagation as in \cite{NakS2} and  \cite{CNLW1}  we have for all $t\ge T_2$
\EQ{
 \Ext(t) \lec \Ext(T_2) \lec S^{2-d} + \de^2 \lec S^{-1}   }
On the other hand, the radial symmetry and the rate of decay of $W$ imply that
\EQ{
 \pt|\CE(T_2)| \lec \sqrt{S}\|\uv\|_E + S\|\uv\|_E^2 \lec \sqrt{S} \delta}
The contribution of $W_{\s(T_3)} (x-c(T_3))$ at $t=T_3$ is estimated as follows. Denote $c_3:= c(T_3), \s_{3}:=\s(T_{3})$. Then
 \EQ{
 \CE(T_3) &= \LR{  xw | e(\vec W_{\s_{3}}(\cdot- c_3))} + \LR{  xw | e(\vec u(T_3)) - e(\vec W_{\s_{3}}(\cdot- c_3) )} \\
 &=: A+B
 }
Now, using \eqref{fav dil},
\EQ{
A &= c_3 \LR{  w | e(\vec W_{\s_{3}}(\cdot- c_3))}    +  \LR{  (x-c_3) w | e(\vec W_{\s_{3}}(\cdot- c_3))}  \\
&= c_3 (E(\vec W)+ o(1)) + \LR{  (x-c_3) w | e(\vec W_{\s_{3}}(\cdot- c_3)), }
}
where $o(1)$ is with respect to the limit $S\to\I$ (uniformly for the other parameters $c_3$, $\s_3$, $T_2$ and $T_3$).
Exploiting~\eqref{FSP} and the obvious cancellation yields
\EQ{
|\LR{  (x-c_3) w | e(\vec W_{\s_{3}}(\cdot- c_3))}  | \lec 1+ \log\big( 1+S^{-1}(T_3-T_2)\big).
}
On the other hand,
\EQ{
B &\lec (T_3-T_2+S) \de.
}
Combining these estimates yields
\EQ{
|c(T_{3})| &\lec | \CE(T_3) - c(T_3) (E(\vec W)+ o(1)) | + | \CE(T_3) -  \CE(T_2)| + | \CE(T_2)|,
}
and therefore
\EQ{\label{c3 bd}
|c(T_3)| 
&\lec   1+S^{-1}(T_3-T_2) +  (T_3-T_2+S) \de+\sqrt{S} \delta.
}
To obtain the desired contradiction, we use the localized virial identity
\EQ{\label{virial}
 \pt V_w(t):=\LR{wu_2|(x\na+d/2)u_1},
 \pr \dot V_w(t)=-K(u_1(t))+O(\Ext(t))=-K(u_1(t))+ O(S^{-1})}
 with the same choice of $w$ as above.
By similar considerations as above, one has the upper bounds
 \EQ{\label{oben1}
 |V_w(T_2)| &\lec  \de \, S^{\frac12}  ,  \\
 |V_w(T_3)| &\lec \de (|c_3| + (T_3-T_2+S)^{\frac12}) + \de^2(T_3-T_2+S)  .}
Setting $S:=\de^{-1}$ in~\eqref{c3 bd} implies
$$ |c(T_3)|\lec 1+ \de (T_3-T_2)$$
and
\EQ{\label{oben2}
 |V_w(T_2)| +  |V_w(T_3)| \lec \de^{\frac12} + \de^2(T_3-T_2) + \de (T_3-T_2)^{\frac12}
}

We claim that integrating the differential equation in~\eqref{virial} and exploiting the ejection dynamics and the variational structure (cf.~\cite{NakS}) leads to the lower bound
\EQ{\label{unten}
 \int_{T_2}^{T_3}\sg K(u_1(t))\,dt \gec \nu(\de,\de_H)\de(T_3-T_2)+\de_H,}
where $0<\nu(\de,\de_H)\to\I$ as $\de\to+0$ and $\de_H$ fixed. This clearly contradicts \eqref{oben2}, provided that we choose $\de_*\ll\de_H^2$ small enough.

It remains to prove~\eqref{unten}.
Let $\cT$ be the set of times at which the distance $d_W(\vec u(t))|_{[T_2,T_3]}$ reaches a local minima in $[\de,\de_S]$.
In particular, $\cT\ni T_2,T_3$ by the choice of $T_2$ and $T_3$.
For every $t_*\in\cT$, we can apply the ejection Lemma \ref{lem: ejection} from $t=t_*$ in both time directions.
Then we get an interval $I(t_*)\subset[T_2,T_3]$ such that $t_*\in I(t_*)$, $d_W(\vec u)$ within $I(t_*)$ is decreasing for $t<t_*$ and increasing for $t>t_*$, and $d_W(\vec u)=\de_H$ on $\p I(t_*)\setminus\{T_2,T_3\}$.
Moreover, imposing
\EQ{ \label{vari cond eps}
 0<\e_* < \e_1(\de_S),}
we can ensure that $\vec u$ stays in $\HL_{\de_S}$ for $t\in[T_2,T_3]$, so that the sign $\sg$ in the ejection lemma is the same for all $I(t_*)$ by Corollary \ref{cor:sign}.
Furthermore, the exponential behavior allows us to estimate
\EQ{ \label{eject est'}
 \pt |I(t_*)|=\int_{t\in I(t_*)}e^{-\s}d\ta \simeq e^{-\s(t_*)}\log(\de_H/d_W(\vec u(t_*))) \le e^{-\s(t_*)}\log(\de_H/\de),
 \pr \int_{I(t_*)}\sg K(u(t))dt \gec \int_{t\in I(t_*)}(e^{k\ta}-C_*)d_W(\vec u(t_*))e^{-\s}d\ta \simeq e^{-\s(t_*)}\de_H.}
Summing this over all $t_*\in\cT$ including $T_2$ and $T_3$, we get
\EQ{ \label{eq:L2}
 \int_{J_1} \sg K(u(t)) dt \gec \de_H + \frac{\de_H/\de}{\log(\de_H/\de)}\de|J_{1}|,\pq J_1:=\Cu_{t_*\in\cT}I(t_*).}
For the remaining times, we have
\EQ{
 \inf_{t\in J_0}d_W(\vec u(t))\ge\de_S, \pq J_0:=[T_2,T_3]\setminus J_1,}
by the definition of $J_1$, so that under \eqref{vari cond eps} we can use the variational bound of Lemma~\ref{lem: variational}.
If $\sg=-1$, then we have
\EQ{\label{eq:L1}
 \int_{J_{0}} \sg \, K(u(t))\, dt \ge \kappa(\de_S)|J_{0}|. }
Adding \eqref{eq:L1} and~\eqref{eq:L2} concludes the $\sg=-1$ case of~\eqref{unten}.

If $\sg=+1$, then the same argument encounters the difficulty that outside of the $\de_{H}$-ball the lower bound of Lemma~\ref{lem: variational} may become degenerate due to smallness of~$\|\na u\|_{2}^{2}$.
Indeed, replacing $\ka(\de_S)$ in the above argument by $\min(\ka(\de_S),c\|\na u\|_2^2)$ and using the uniform bound on $\|\vec u\|_\HH$ in the region $\Sg=+1$ yields
\EQ{
\label{inter bd}
\int_{T_2}^{T_3}  K(u_1(t))\,dt \gec \frac{\de_H/\de}{\log(\de_H/\de)}\de \int_{T_2}^{T_3} \| \na u(t)\|_{2}^{2}\,dt + \de_H. }
This leads to \eqref{unten} for $\sg=+1$ if
\EQ{ \label{inter bd1}
\int_{T_2}^{T_3} \| \na u(t)\|_{2}^{2}\,dt \gtrsim T_{3}-T_{2}. }
Therefore assume that
\EQ{\label{klein}
 \int_{T_2}^{T_3} \| \na u(t)\|_{2}^{2}\,dt \le \kappa_{2}( T_{3}-T_{2})}
with some absolute constant $\kappa_{2}$.
To lead~\eqref{klein} to a contradiction, we use the (localized) energy equipartition
\EQ{\label{vir 2}
 \p_t\LR{wu_t|u}\pt=\|\dot u(t)\|_2^2-K(u(t))+O(\Ext(t))
 \ge 2E(\vec u)-\|\na u\|_2^2 + O(\de),
}
where $w$, $\Ext$ and $S=\de^{-1}$ are as before.
Taking $\de_*,\ka_2\ll J(W)$, we obtain
\EQ{\label{fertig}
 [\LR{wu_t|u}]_{T_{2}}^{T_{3}} \ge E(\vec u)(T_{3}-T_{2}).
}
On the other hand, the same argument as for \eqref{oben2} yields
\EQ{
|\LR{wu_t|u}(T_{3}) - \LR{wu_t|u}(T_{2})| \lec \de^{1/2} + \de^2(T_3-T_2) + \de(T_3-T_2)^{1/2},}
which contradicts~\eqref{fertig} since $T_{3}-T_{2}\gg 1 \gg \de$.
\end{proof}

The above result has some important implications for the sign functional from Corollary~\ref{cor:sign}. To be specific, let
\EQ{
 \pt \HH_* = \{\U\fy\in\HH \mid E(\U\fy)\le J(W)+\e_*^2\},
 \pr \HH_X = \{\U\fy\in\HH_* \mid E(\U\fy)< J(W)+d_W^2(\U\fy)/2\}.}
It is easy to see that $\HH_*\setminus \HH_X$ is in a small neighborhood of $\pm\Sm_0$.

\begin{cor} \label{cor: onepass}
The sign function $\Sg$ in Corollary \ref{cor:sign} is continuous on $\HH_X$, and has the following properties.
\begin{enumerate}
\item Every solution $u$ in $\HH_*$ can change $\Sg(\vec u(t))$ at most once. Moreover, it can enter or exit the region $d_W(\vec u)<\de_*$ at most once.
\item The region $\Sg=+1$ is bounded in $\HH$, while the region $\Sg=-1$ is unbounded.
\item If $\U\fy\in\HH_X$ and $E(\U\fy)\le J(W)+\e_1^2(d_W(\U\fy))$, then $\Sg(\U\fy)=\sign K(\fy_1)$, with the convention $\sign 0=+1$.
\item If $\U\fy\in\HH_X$ and $d_W(\U\fy)\le \de_S$, then $\Sg(\U\fy)=-\sign\la_{1}(\fy_1)$.
\end{enumerate}
\end{cor}
The proof is the same as in the radial case \cite{CNLW1}, so we omit it.
Note that $\HH_*\setminus\HH_X$ is included in $d_W<\de_*$, and that (3)--(4) completely determine $\Sg(\U\fy)$, since we have chosen $\e_*<\e_1(\de_S)$ in \eqref{vari cond eps}. Moreover, $\Sg(\U\fy)$ depends only on $\fy_1$.

It remains to determine the fate of the solutions in $\HH_*$ with $d_W\ge\de_*$. We will do this in the following two sections for $\Sg=\pm1$ , respectively.

\section{Blow-up after ejection }

In analogy to \cite{CNLW1}, we now prove\footnote{The proof is essentially the same as in \cite{CNLW1}, but since we employ somewhat different notation, we provide the details for the reader's convenience.} the following

\begin{prop}\label{prop:blowup} No solution $\vec{u}\in \HH_*$ satisfying the conditions in Theorem~\ref{thm: onepass}  can stay strongly continuous with respect to the topology of $\HH_*$ and satisfy the requirements
\EQ{
 d_W(\vec{u})\geq \delta,\pq \Sg = -1,}
for all $t>T_2$.
\end{prop}
\begin{proof}
Suppose, for the sake of contradiction, that the solution actually exists on $(0,\infty)$.
Write $w = \chi(\frac{|x|}{t+R})$ for some $R\gg 1$ to be chosen, and $\chi\in C_0^\infty(\R)$ a non-negative cutoff function, with $\chi'(r)\leq 0$ for $r\geq 0$ and $\chi(r) = 1$ on $r\leq 1$. Also introduce
\EQ{
y(t) = \langle w u | u\rangle.}
Then we have
\begin{equation}\label{eqN:doty}
\dot{y}(t) = \langle \dot{w}u + 2w\dot{u} | u\rangle \geq 2\langle w\dot{u} | u\rangle
\end{equation}
Writing $\Ext(t):=\|\vec u(t)\|_{\HH(|x|>t+R)}^2$,
we find using Hardy's inequality
\begin{align}\ddot{y} &= \langle 2w| \dot{u}^2 - |\nabla u|^2 + |u|^{2^*}\rangle + \langle\ddot{w}u|u\rangle + \langle 4\dot{w}u|\dot{u}\rangle + 2\langle u\nabla w|\nabla u\rangle\\
&=2(\|\dot{u}\|_{2}^2 - K(u)) + O(\Ext(t))
\end{align}
It follows from the finite propagation as before that we can choose $\tau$ large enough such that $\Ext(t)\ll \eps_{*}^2$ for all $t>0$.

We next follow the argument in the proof of Theorem~\ref{thm: onepass}, especially the part after \eqref{unten}. Thus with $T_2$ as in the proof of Theorem~\ref{thm: onepass}, we write
\EQ{
 [T_2,\infty) = J_0\cup J_1}
with $J_0$ and $J_1$ defined just as in the proof of Theorem~\ref{thm: onepass}, with $T_3$ replaced by $+\infty$.
Then as before we find $-K(u)>\kappa(\delta_M)$ on $J_0$; on the complement
\EQ{
 J_1=\Cu_{t_*\in\cT}I(t_*),}
we also obtain the lower bound
\EQ{
  -\int_{I(t_*)}K(u(t))\,dt \gg \de|I(t_*)|.}
We conclude that
\EQ{
\lim_{t\rightarrow+\infty}\dot{y}(t)  = \lim_{t\rightarrow+\infty}y(t) =  +\infty}
and $y(t)$ is increasing for large enough $t$.

Next, write
\begin{equation}\label{eqn:blowup1}
\|\dot{u}\|_{2}^2 - K(u) = (1+\frac{2^*}{2})\|\dot{u}\|_{2}^2 + \frac{2^*-2}{2}\|\nabla u\|_{2}^2 - 2^* E(\vec{u})
\end{equation}

If $t\in J_0$, then from the variational characterization of $W$ for $K<0$, we have
\EQ{
 \|\nabla u(t)\|_2>\|\nabla W\|_2 }
and so
\[
E(\vec{u})<J(W)+\eps_*^2 = \frac{2^* - 2}{2\cdot 2^{*}}\|\nabla W\|_{2}^2+\eps_*^2<\frac{2^* - 2}{2\cdot 2^{*}}\|\nabla u(t)\|_{2}^2+\eps_*^2,
\]
which in conjunction with \eqref{eqn:blowup1} implies
\begin{equation}\label{eqn:blowup2}
\|\dot{u}\|_{2}^2 - K(u)>(1+\frac{2^*}{2})\|\dot{u}\|_{2}^2 - 2^*\eps_*^2.
\end{equation}
We also have the bound from Lemma \ref{lem: variational}
\begin{equation}\label{eqn:blowup3}
 \|\dot{u}\|_{2}^2 - K(u)> \|\dot{u}\|_{2}^2 + \kappa(\delta_S),
\end{equation}
which on account of \eqref{vari cond eps} and interpolation with the bound \eqref{eqn:blowup2} implies
\begin{equation}\label{eqn:blowup4}
\ddot{y}(t)>4(1+c)\|\dot{u}(t)\|_{2}^2 + 2\e_*^2,
\end{equation}
provided $t\in J_0$, where $c>0$ (e.g.~$c=1/(2(d-1))$).

Next, we consider the case when $t\in J_1$. We use the following general inequality in $\dot H^1(\R^d)$, \cite[Lemma 5.2]{CNLW1}:
\EQ{
 \|\nabla W\|_{2}^2\leq \|\nabla u\|_2^2 + \frac{d-2}{2}K(u) + O(K^2(u)/\|\nabla u\|_2^2).}
As $\|\nabla u(t)\|_2^2\simeq \|\nabla W\|_2^2$ for $t\in J_1$, it follows that
\EQ{
 E(\vec{u})<J(W)+\eps_*^2 \leq \frac{2^{*} - 2}{2\cdot 2^{*}}\|\nabla u\|_2^2 + \frac{d-2}{2d}K(u) + O(K^2(u)+\eps_*^2). }
Then \eqref{eqn:blowup1} implies that if $t\in J_1$,
\begin{equation}\label{eqn:blowup5}
 \ddot{y}>4(1+c)\|\dot{u}\|_2^2 - 2K(u) - O(K^2(u)+\eps_*^2).
\end{equation}

To obtain a contradiction, we next observe that
\EQ{
 \LR{\dot wu|u}=-\LR{x\cdot\na w||\ti w u|^2}/(t+R)
 =\LR{wu|\ti w(r\p_r+d/2)\ti w u}/(t+R),}
where $\ti w:=\ti\chi(|x|/(t+R))$ with some $\ti\chi\in C^\I(\R)$ satisfying $\ti\chi=1$ on $\supp\chi'$ and $\ti\chi(r)=0$ for $|r|\le 1$. Hence
\EQ{
 |\dot{y}| \le \|wu\|_2\|2\dot u+\ti w(r\p_r+d/2)\ti wu/(t+R)\|_2}
and so
\EQ{
 |\dot y|^2/y \pt\le \|2\dot u+\ti w(r\p_r+d/2)\ti wu/(t+R)\|_2^2
 \pn\le 4\|\dot u\|_2^2 + O(\Ext(t)).}
We then infer from \eqref{eqn:blowup4} that for $t\in J_0$,
\begin{equation}\label{eqn:blowup6}
\ddot{y}(t)\geq (1+c)\frac{\dot{y}(t)^2}{y(t)}+\eps_*^2,
\end{equation}
and from \eqref{eqn:blowup5} that for $t\in J_1$,
\begin{equation}\label{eqn:blowup7}
 \ddot{y}(t)\geq (1+c)\frac{\dot{y}(t)^2}{y(t)}-K(u(t)) - O(K^2(u(t))+\eps_*^2).
\end{equation}
Now consider
\EQ{
  \partial_t^2 y^{-c} = -cy^{-1-c}\left[\ddot{y} - (1+c)\frac{\dot{y}(t)^2}{y(t)}\right].}
Again from the asymptotic behavior of $K(u)$ on each $I(t_*)$ in $J_1$ given by Lemma~\ref{lem: ejection}, \eqref{eqn:blowup7}, and the fact that $y^{-1-c}$ is decreasing for large enough $t$ imply that
\EQ{ \label{disc dec}
 \int_{J(t_*)\cap(-\I,T)}y(t)^{-1-c}[-K(u(t)) - O(K^2(u(t))+\eps_*^2)]\,dt<0,}
for large $t_*>T_2$ and any $T>t_*$.  In particular, we infer that
\EQ{\label{eqn:blowupcrux}
 -\p_ty^{-c}(t_*) \ge \inf_{t\in I(t_*)}\p_ty^{-c}(t),}
while $\p_ty^{-c}(t)$ is decreasing in $J_0$.
Hence for some $t_*\in\cT$ and for all $t>t_*$ we have
\EQ{
 \p_ty^{-c}(t)\le \p_ty^{-c}(t_*)<0.}
This implies finite time blow up, contradicting the earlier assumption.
\end{proof}

\section{Scattering after ejection}
Here we essentially repeat the argument given in \cite{CNLW1} for the reader's convenience, with the small changes necessitated  by the presence of space  and momentum translations.
In the region $\Sg=+1$, we already know that all solutions are uniformly bounded in $\HH$, but it is not sufficient for global existence of strongly continuous solution in the critical case. Now we resort to the recent result by Duyckaerts-Kenig-Merle \cite{DKM2} to preclude concentration (type II) blow-up. This is the only place where we have to restrict the dimensions\footnote{Strictly speaking, the long-time perturbation argument should be also modified for $d>6$ in the scattering proof of Proposition \ref{prop:Sg+ scatt}, but it is a minor issue. See \cite{CNLKG,ScatBlow} for the solution.} to $3$ or $5$

\begin{prop} \label{prop:Sg+ global}
No solution as in Theorem~\ref{thm: onepass}
blows up in $\HH_X$ with $\Sg=+1$ in the region $t\geq T_2$.
\end{prop}
\begin{proof}
First,   Lemma~\ref{lem: ejection} precludes blow-up in the hyperbolic region, since the scaling parameter is a priori bounded during the ejection process, which is valid when reversing the time direction. Hence a blow-up may happen only when $d_W(\vec u(t))>\de_H$, where $K(u(t))\ge 0$ and so the energy assumption in Theorem~\ref{thm: onepass}  implies
\EQ{
 \frac{\|\dot u(t)\|_2^2}{2}+\frac{\|\na u(t)\|_2^2}{d}=E(\vec u)-\frac{K(u(t))}{2^*} < J(W)+\eps_*^2 = \frac{\|\na W\|_2^2}{d}+\eps_*^2.}
This allows us to employ the main result in \cite{DKM2}, after reducing $\eps_*$ if necessary.
Suppose $u$ is a solution on $[0,T_+)$ in $\HH_X$ with $\Sg=+1$ and $d_W(\vec u(t))>\delta_H$ with the blow-up time $T_+<\I$.
According to their result, we can then write for $t$ sufficiently near $T_+$
\EQ{
 \vec u(t) = \vec W_{\s(t)}(x-c(t)) + \U\fy + o(1) \pq\IN\HH,}
for some $\s(t)\to\I$, $c(t)\in\R^d$ and some fixed $\U\fy\in\HH$.
It is then easily checked that as $t \to T_+-0$ we have
\EQ{ \label{eq:Kv}
 K(u(t)) = K(W_{\s(t)}) + K(\fy_1) + o(1) = K(\fy_1) +o(1),}
from which we infer in particular that $K(\fy_1)\ge 0$. Similarly, we obtain
\EQ{ \label{eq:Ev}
 J(W)+\eps_*^2> E(\vec u) = J(W) + E(\U\fy),}
which implies via $K(\fy_1)\ge 0$,
\EQ{
 \|\fy_2\|_2^2/2 + \|\na\fy_1\|_2^2/d = E(\U\fy)-K(\fy_1)/2^*< \eps_*^2.}
This however contradicts $d_W(\vec u(t))>\de_H \gg \eps_*$ near $T_+$.
\end{proof}

Next we employ the Kenig-Merle scheme from \cite{KM1,KM2} to improve the above result. The one-pass theorem will be incorporated in the same way as in the subcritical case \cite{NakS}. Extinction of the critical element requires a little extra work due to the possibility of concentration, which will be however reduced to the above proposition.
\begin{prop} \label{prop:Sg+ scatt}
Every solution staying in $\HH_X$ with $\Sg=+1$ and $d_W\ge\delta_*$ for $t>0$ scatters to $0$ as $t\to+\I$ with uniformly bounded Strichartz norms on $[0,\I)$.
\end{prop}
The restriction $d_W\ge\delta_*$ is essential for the uniform Strichartz bound, since the latter does not hold for all scattering solutions, even for $E(\vec u)<J(W)$.
\begin{proof}
We argue by contradiction. Let $u_n$ be solutions on $[0,\I)$ in $\HH_X$ satisfying
\EQ{\label{eq:unseq}
 \pt E(\vec u_n) \to E_{*}\le J(W)+\eps_*^2,\quad \|u_{n}\|_{L^q_{t,x}(0,\I)} \to \I,
 \pr d_W(\vec u_n(t))\ge \delta_*, \pq \Sg(\vec u_n(t))=+1, \pq (t>0)
}
where we choose $q=2(d+1)/(d-2)$ so that $L^q_{t,x}$ is an admissible Strichartz norm for the wave equation on $\R^d$.
Here and after, $X(I)$ denotes the restriction to $I\times\R^d$ of the Banach function space $X$ on $\R\times\R^d$. It is well-known that $L^q_{t,x}$ and the energy norm are sufficient to control all the other Strichartz norms, such as $L^p_t \dot B^{1/2}_{p,2}$ with $p=2(d+1)/(d-1)$, as well as the nonlinear term in some dual admissible norm such as in $L^{p'}_t \dot B^{1/2}_{p',2}$ (see, for example, \cite{GSV}).

We may assume that $E_*$ is the minimum for the above property.
Following the Kenig-Merle argument, the proof consists of two parts: construction and exclusion of a critical element.

\medskip\noindent{\bf Part I}: Construction of a critical element.

Assuming the existence of \eqref{eq:unseq}, we are going to show that there is a critical element $u_*$, that is a solution on $[0,\I)$ in $\HH_X$ satisfying
\EQ{ \label{def crit}
 \pt E(\vec u_*) = E_*,\pq \|u_*\|_{L^q_{t,x}(0,\I)}=\I,
 \pq d_W(\vec u_*(t))\ge \de_*, \pq \Sg(\vec u_*(t))=+1,}
and that its trajectory is precompact modulo dilations and translations in $\HH$.

If $d_W(\vec u_n(0))<\delta_H$, then by Lemma~\ref{lem: ejection}, we have $d_W(\vec u_n(t))\ge \de_H$ at some later $t>0$. Since the Strichartz norm on the ejection time interval is uniformly bounded, we may time-translate each $u_n$ so that
\EQ{
 d_W(\vec u_n(0))\ge\delta_H,}
without losing \eqref{eq:unseq}.

Since we chose $\eps_*<\e_1(\de_S)\le\eps_1(\de_H)$, Lemma \ref{lem: variational} implies
\EQ{\label{eq:Klower}
 K(u_{n}(0)) \ge \min(\kappa(\delta_H),c\|\nabla u_{n}(0)\|_{2}^{2}).}
Now apply\footnote{In what follows, we will pass to subsequences without any further mention.} the Bahouri-G\'erard decomposition from~\cite{BaG}, see also Lemma~4.3 in \cite{KM2}, to $\{\vec u_n(0)\}_{n\ge1}$. Let $U(t)$ denotes the free wave propagator. We conclude that there exist $\lam^j_n>0$, $t^j_n\in\R$, $x^j_n\in \R^d$, $\U\fy^j\in\HH$ and free waves $w^J_n$ such that for any $J\ge1$
\EQ{ \label{eq:BG}
  \pt U(t)\vec u_{n}(0) = \sum_{j=1}^{J} \vec V^j_n(t) + \vec w^J_n(t),
  \pq \vec V^j_n(t):=U(t+t^j_n)T^j_n\U\fy^j,}
where $T^j_n:=T^{-x^j_n}\vec S^{\log\la^j_n}$, such that
\EQ{ \label{param orth}
 \pt |\log(\lam^j_n/\lam^k_n)|+\lam^j_n|t^j_n-t^k_n| + \lam^j_n|x^j_n-x^k_n| \to\I }
for each $j\not=k$,
\EQ{ \label{ene orth}
 \pt \lim_{n\to\I} \Bigl[ \|\vec u_{n}(0) \|_\HH^{2} -\sum_{j=1}^{J}  \|\vec V^j_n(0)\|_\HH^{2} - \|\vec w^J_n(0)\|_\HH^{2} \Bigr]=0,
 \pr \lim_{n\to\I} \Bigl[ E(\vec u_{n}(0)) - \sum_{j=1}^{J} E(\vec V^j_n(0)) - E(\vec w^J_n(0)) \Bigr] =0}
for each $J$, and
\EQ{ \label{Str vanish}
  \lim_{J\to\I}\limsup_{n\to\I}\|w^J_n\|_{L^\I_t L^{2^*}_x(\R)\cap L^q_{t,x}(\R)}=0.}
The last property applies to any other non-sharp Strichartz norm by interpolation, since those free waves are all uniformly bounded.

First we check that all components retain the property $K\ge 0$ at $t=0$. Indeed, one has
\EQ{
 E(\vec u_n)  - \frac{1}{2^*} K(u_{n}(0)) &\ge \frac 1d\|\vec u_n(0)\|_\HH^2
 = \sum_{j=1}^{J} \frac 1d \|\vec V^j_n(0)\|_\HH^{2} + \frac 1d \|\vec w^J_n(0)\|_\HH^{2} +o(1),}
where $o(1)\to 0$ as $n\to\I$.
Hence if $\|\na u_n(0)\|_2^2\lec\eps_*^2$, then $\|\na V^j_n(0)\|_2^2\lec\eps_*^2\ll 1$, and so $K(V^j_n(0))\ge 0$.
Otherwise, the lower bound in \eqref{eq:Klower} is much bigger than $\eps_*^2$, so for large $n$, we conclude from the above inequality  that
\EQ{
  \frac{\|\vec V^j_n(0)\|_\HH^2}{d} <  J(W),}
which implies $K(V^j_n(0)) \ge 0$, by the variational property of $W$. The same argument implies $K(w^J_n(0))\ge 0$ as well. Thus, each component has non-negative energy $E$.

Now let $U^j$ be the nonlinear profile associated with $V^j_n$, that is the nonlinear solution satisfying as $n\to\I$,
\EQ{ \label{eq:UjVj}
 \|\vec U^j(s^j_n)-U(s^j_n)\U\fy^j\|_\HH \to 0, \pq s^j_n:=\lam^j_n t^j_n,}
defined uniquely around $t=s^j_\I:=\lim_{n\to\I}s^j_n$, such that
\EQ{
 \|\vec U^j_n(0)-\vec V^j_n(0)\|_\HH\to 0\pq \vec U^j_n(t):=(T^j_n\vec U^j)(\lam^j_n(t+t^j_n)).}
By the scaling invariance of the equation, each $U^j_n$ is also a solution, defined locally around $t=0$. Hence the above property of $\vec V^j_n(0)$ is transferred to $U^j_n$:
\EQ{ \label{est on profiles}
 \pt K(U^j_n(0))\ge 0, \pq 0\le E(\vec U^j_n)=E(\vec U^j) \simeq \|\vec U^j_n(0)\|_\HH^2,
 \pr \sum_{j=1}^J E(\vec U^j)+\lim_{n\to\I}E(\vec w_n^J(0))=E_*.}
We may assume that $j=1$ gives the maximum among $E(\vec U^j)$, then by \cite{KM2}, each $U^j$ for $j>1$ exists globally and scatters with
\EQ{
 \sum_{j=2}^J \|U^j\|_{L^q_{t,x}(\R)}^2 \lec \sum_{j=2}^J E(\vec U^j) \le \frac{2}{3}J(W).}

Now assume $\|U^1\|_{L^q_{t,x}(\R)}<\I$, which is the case if $E(U^1)<J(W)$.
Then from the long-time perturbation theory, cf.~Theorem~2.20 in~\cite{KM2},
one obtains the {\em nonlinear profile decomposition}  for the solutions $u_{n}(t)$, provided $J$ is large and fixed,
and $n\ge n_0(J)$ is sufficiently large:
\EQ{\label{eq:nonlinearBG}
 \pt u_{n} = \sum_{j=1}^{J} U^j_n+ w^J_n  + R^J_n,
 \pr \lim_{J\to\I}\limsup_{n\to\I} \left[\|\vec R^J_n\|_{L^\I_t\HH(\R)}+\|R^J_n\|_{L^q_{t,x}(\R)}\right]= 0, }
which implies $u_n$ is bounded in $L^q_{t,x}$, contradicting \eqref{eq:unseq}. Thus we have obtained
\EQ{
 \| U^1\|_{L^q_{t,x}(T_-,T_+)}=\I, \pq J(W)\le E(\vec U^1)\le E_*, \pq \sum_{j=2}^JE(\vec U^j)+\|\vec w^J_n\|_\HH^2 \lec \e_*^2,}
where $(T_-,T_+)$ is the maximal existence interval of $U^1$.

We now distinguish three cases (a)--(c) by $s^1_\I=\lim_{n\to\I}\lam^1_nt^1_n$:

\medskip\noindent{\bf (a)} $s^1_\I=\I$. Then by definition \eqref{eq:UjVj}, $U^1$ is a local solution around $t=\I$ with finite Strichartz norms, and
\EQ{
 \|U^1_n\|_{L^q_{t,x}(0,\I)}=\|U^1\|_{L^q_{t,x}(s^1_n,\I)}\to 0.}
Hence we can use the long-time perturbation argument on $(0,\I)$, which gives a contradiction via~\eqref{eq:nonlinearBG} as above.

\medskip \noindent{\bf (b)} $s^1_\I=-\I$.  In this case $U^1$ scatters at $t=-\I$ by definition.

If $d_W(\vec U^1(t))>\delta_*/2$ for all $t<T_+$, then $\vec U^1(t)$ remains in the region $\Sg=+1$ from $t=-\I$. Hence $T_+=\I$ by Proposition \ref{prop:Sg+ global},
and $\|U^1\|_{L^q_{t,x}(0,\I)}=\I$. Moreover, Theorem~\ref{thm: onepass} together with Lemma~\ref{lem: ejection} implies that $d_W(\vec U^1(t))\ge\de_*$ for large $t$.
Hence $U^1$ is a critical element after some time translation.

Otherwise, $d_{W}(\vec U^1(t_{*})) = \delta_*/2$ at some minimal $t_{*}<T_{+}$, until which time $U^1$ remains in $\Sg=+1$, and $\| U^1 \|_{L^q_{t_x}(-\I,t_{*})} <\I$.
Hence one can apply the nonlinear profile decomposition on the interval $\lam^1_n(t+t^1_n)\le t_*$ as in \eqref{eq:nonlinearBG}, which yields in particular that, upon choosing $J$ sufficiently large,
\EQ{\label{eq:dSest}
 d_{W}(\vec u_{n}((t_{*}-s^1_n)/\lam^1_n))
 \le d_{W}(\vec U^1(t_{*})) + O(\eps_*) + o(1) \le \frac{2}{3}\delta_* +o(1),}
as $n\to\I$.  However, since $t_{*}-s^1_n\to\I$, this contradicts our assumption
\EQ{
 \inf_{t\ge 0}d_{W}(\vec u_{n}(t))\ge \delta_*.}

\medskip \noindent{\bf (c)} $s^1_\I\in\R$.
Then by the same perturbative arguments as above, the nonlinear profile decomposition~\eqref{eq:nonlinearBG} holds on any compact interval in $(T_-,T_+)/\lam^1_n-t^1_n$.
Thus, as in the case (b), we deduce from $\inf_{t\ge 0}d_{W}(\vec u_{n}(t))\ge\delta_*$ that
\EQ{
 \inf_{s^1_\I\le t<T_+} d_{W}(\V U^1(t)) \ge \delta_*/2.}
Then the same argument as in (b) implies that $T_+=\I$ and $\|U^1\|_{L^q_{t,x}(s^1_\I,\I)}=\I$, since otherwise $U^1$ scatters and the nonlinear profile decomposition holds on $[0,\I)$, contradicting \eqref{eq:unseq}.

\medskip
Thus we arrive at the conclusion that $s^1_\I<\I$ and $U^1$ is a critical element after time translation.
This implies $E(U^1)=E_*$ by the minimality, which extinguishes the other profiles $U^j$ ($j>1$) as well as the remainder $w^J_n$ as $n\to\I$, through the nonlinear energy decomposition.

Having constructed a critical element $u_*$, we apply the above argument to the sequence
\EQ{
 u_n(t)=u_*(t-t_n), \pq t_n\to\I.}
Then the vanishing of all but one (free) profile implies that for some continuous $\s(t)\in\R$, $x(t)\in \R^d$,
\EQ{\label{eq:compact}
 \{T^{x(t)}\vec S^{\s(t)}\vec u_*(t)\}_{t\ge 0} \subset \HH}
is precompact, concluding the first part of the proof.

Before proceeding to the extinction, we show that
\EQ{
 |P(\vec u)|\lec\e_*.}
Suppose towards a contradiction that $|P(\vec u_*)|=|P_1(\vec u_*)|\gg\e_*$.
Then, since $J(W)\le E(\vec u_*)<J(W)+\e_*^2$, we can use the Lorentz transform to reduce the energy below $J(W)$.
Indeed, let $w$ be any global strong energy solution of~\eqref{main PDE} and Lorentz transform it as follows: with a parameter $\nu\in\R$,
\EQ{ \label{Lorentz}
 w(t,x) \mapsto w_\nu(t,x):=w(t\cosh\nu+x_1\sinh\nu,x_1\cosh\nu+t\sinh\nu,x_2,\ldots).}
Then one checks that  $w_\nu$ is again a strong energy solution of~\eqref{main PDE} which satisfies
\EQ{\label{eq:EPtrans}
 \pt E( w_\nu) = E( w)\cosh\nu+P_1( w)\sinh\nu,
 \pr P_1( w_\nu)=P_1( w)\cosh\nu+E( w)\sinh\nu, \pq P_\al( w_\nu)=P_\al( w),\ (\al>1).}
Now we claim that we can apply the above transform to the forward global solution $u_*$, and then with some $\nu=O(\e_*)$ we can construct another forward global solution $u_\star$ with $E(\vec u_\star)<J(W)$ and $\|u_\star\|_{L^q_{t,x}(0,\I)}=\I$. This contradicts Kenig-Merle's result \cite{KM2} for $E<J(W)$.
In order to transform a solution with infinite Strichartz norm, we argue in the same way as in the subcritical case using the finite propagation speed:
\begin{lem} \label{global Lorentz}
Let $u$ be a solution of \eqref{main PDE} in $C(I;\HH)\cap L^q_{t,x}(I)$ on a time interval $I\ni T$. Then there is an open neighborhood $O$ of the identity in the Lorentz group, such that the transform of $u$ by any $g\in O$ extends to a solution (with finite energy and $L^q$) in a space-time region including a time slab which contains $T$. If $u$ is a solution in $C([T,\I);\HH)\cap L^q_{loc}((T,\I;L^q_x)$, then for any Lorentz transform $u'$ of $u$, there exists $T'\in\R$ such that $u'$ extends to a solution in $C([T',\I);\HH)\cap L^q_{loc}((T',\I);L^q_x)$. Moreover, if $\|u'\|_{L^q_{t,x}(T',\I)}<\I$ then $\|u\|_{L^q_{t,x}(T,\I)}<\I$.
\end{lem}
The proof is also the same as for \cite[Lemmas 6.1 and 6.2]{NakS2}, so we omit it.

\medskip\noindent{\bf Part II}: Exclusion of a critical element.

Let $u_*$ be a critical element \eqref{def crit}, hence
\EQ{
 \vec w_*(t):=\rho(t)^{d/2}\vec u_*(t,\rho(t)(x-x(t))), \pq \rho(t)=e^{-\s(t)}}
for $t\ge 0$ is precompact in $\HH$. We proceed in three steps.

\medskip\noindent{\bf Step 1}: $\limsup_{t\to\infty}\rho(t)/t < \infty$.
To see this, note that by finite propagation speed, we have
\EQ{
\lim_{R\to\infty}\sup_{t\geq 0} \|\V u_*(t)\|_{\HH(|x|> t+R)} = 0,}
whence we have
\EQ{ \label{compact w*}
\lim_{R\to\infty}\sup_{t\geq 0} \|\V w_*(t)\|_{\HH(|x-x(t)|> (t+R)/\rho(t))} = 0.}
If for some sequence of times $\{s_n\}_{n\geq 1}$ we had $\rho(s_n)/s_n\to \infty$, then by pre-compactn-\\ ess of $\{\V w_*(t)\}_{t\geq 0}$, we get $\|\vec w_*(s_n)\|_{\HH}\to 0$, whence also $\|\vec u_*(s_n)\|_{\HH}\to 0$, which would force $E_* = 0$, a contradiction.

\medskip\noindent{\bf Step 2}: $\liminf_{t\to \infty}\rho(t)/t >0$.
This follows from the localized virial identity \eqref{virial}, together with the control on the energy center as well as the energy equipartition, as in the proof of Theorem \ref{thm: onepass}. By the precompactness, there is $R>0$, depending on $u_*$, such that for all $t\ge 0$
\EQ{ \label{compact in x}
 \|\vec u_*(t)\|_{\HH(|x+\rho(t)x(t)|>R\rho(t))}< \e_*. }
Suppose for contradiction that $\liminf_{t\to\I}\rho(t)/t= 0$.
Choose $T_3\gg T_2\gg 1$ and $R_2,R_3>0$ such that
\EQ{ \label{cond TRj}
 \rho(T_j)\ll \e_* T_j/R, \pq R_j:=R\rho(T_j).}
Let $c_j:=-\rho(T_j)x(T_j)$ with $c(T_2)=0$ by space translation. Then \eqref{compact in x} implies in particular that $\|\vec u_*(T_j)\|_{\HH(|x-c_j|>R_j)} < \e_*$,
hence by the finite propagation speed we have
\EQ{
 |c_3|=|c_3-c_2|\le |T_3-T_2|+R_2+R_3=|T_3-T_2|+O(\e_*T_3),}
where we used \eqref{cond TRj}. Let
\EQ{
 w(t,x)=\chi(|x|/(t-T_2+R_2)), \pq \Ext(t)=\|\vec u_*(t)\|_{\HH(|x|>t-T_2+R_2)}^2,}
with the same $\chi$ as in \eqref{def w}. Using the finite propagation speed as before, one has
\EQ{
 \sup_{t>T_2}\Ext(t)\lec \e_*^2.}
For the localized center of energy $\CE(t)=\LR{xw|e(\vec u)}$, we infer as before
\EQ{
 \pt \dot\CE(t)=O(\Ext(t))-P(\vec u)=O(\e_*),
 \pq |\CE(T_2)|\lec  R_2 \ll \e_*T_2,
 \pr \CE(T_3)= \int_{|x-c_3|<R_3}xe(\vec u_*(T_3))dx+O(\e_*^2T_3)
 =c_3E(u_*)+O(\e_*T_3).}
Thus we obtain upon integrating $\CE(t)$ on $(T_2,T_3)$,
\EQ{
 |c_3| \lec \e_*T_3.}
For the localized virial $V_w(t)=\LR{w\dot u_*|(r\p_r+d/2)u_*}$, one has as before
\EQ{
 \pt \dot V_w(t)=-K(u_*(t))+O(\Ext(t))\lec -K(u_*(t))+O(\e_*^2),
 \pr |V_w(T_2)|\lec R_2 \ll \e_* T_2,
 \pq |V_w(T_3)|\lec |c_3|+R_3 \lec \e_*T_3,}
Integrating $\dot V_w$ on $(T_2,T_3)$, and then arguing as for \eqref{inter bd}, we obtain
\EQ{
 \de_*\int_{T_2}^{T_3}\|\na u_*(t)\|_2^2dt - O(\de_S) \lec \int_{T_2}^{T_3}K(u_*(t))dt \lec \e_* T_3,}
where the negative term $-O(\de_S)$ arises in case\footnote{This could be avoided by taking $\e_*<\e_1(\de_*)$ instead of $<\e_1(\de_S)$.} a hyperbolic interval $I(t_*)\not\subset(T_2,T_3)$ has only its negative part inside $(T_2,T_3)$.

Similarly, using Hardy, we have as before
\EQ{
 \pt \p_t\LR{w\dot u_*|u_*}\ge 2E(\vec u_*)-\|\na u_*(t)\|_2^2+O(\e_*^2),
 \pq |[\LR{w\dot u_*|u_*}]_{T_2}^{T_3}|\lec \e_*T_3.}
Integrating the left inequality and combining it with the above estimates, we obtain
\EQ{
 E(\vec u_*)T_3 \lec \e_*\de_*^{-1}T_3 + \e_*\de_S/\de_*.}
Thus we arrive at $J(W)\le E(\vec u_*)\lec \e_*/\de_* \ll J(W)$, a contradiction.

It now also follows that we may put $x(t) = 0$ for all $t\ge 0$, since the assumption
\EQ{
\limsup_{t\rightarrow\infty}|x(t)| = \infty}
contradicts the compactness property of $\vec w_*$, see \eqref{compact w*}.

\medskip\noindent{\bf Step 3}: Construction of a blow up solution via a re-scaling of $u_*$. Pick a sequence $s_n\to \infty$ with
$\lim_{n\to\infty}\rho(s_n)/s_n = c\in (0, \infty)$, as well as $\V w_*(s_n)\to \exists\U\fy$ in $\HH$. Define a sequence of solutions
\EQ{
 u_n(t, x): = s_n^{d/2-1}u_*(s_n t, s_n x)}
whence we have $\Vec u_n(1)\to c^{-d/2}\U\fy(x/c)$ in $\HH$.

The above two steps imply that $\Vec u_n$ is precompact in $C([\tau,1];\HH)$ for any $0<\tau<1$, and so, after replacement by a subsequence, it converges to some $\Vec u_\I$ in $C((0,1];\HH)$. By the local wellposedness theory, it has finite Strichartz norms locally in time, and so $u_\I$ is the unique strong solution on $(0,1]$ with the initial condition $\Vec u_\I(1)=\U\fy$.
Clearly we also have $d_W(\Vec u_\infty(t))\geq \delta_*$ and $\Sg(\V u_\I(t))=+1$ for $0<t\le 1$.

We now show that $u_\infty$ is a solution blowing up at $t=0$, which contradicts Proposition \ref{prop:Sg+ global}.
The fact that $u_\infty$ blows up at $t=0$ follows from the next

\noindent{\bf Claim}: \emph{$u_\infty(t, x) = 0$ on $|x|>t$.
To see this, pick $0<\eps\ll 1$ arbitrary, let $m$ large enough such that $\|\V w_*(s_m) - \U\fy\|_{\HH} \ll\eps$ and further pick $R>0$ such that $\|\U\fy\|_{\HH(|x|>R)}\ll \eps$.
Then for $n>m$, we have
\EQ{
 \|\vec u_n(s_m/s_n)\|_{\HH(|x| > R\rho(s_m)/s_n)} = \|\vec w_*(s_m)\|_{\HH(|x|>R)}\ll \eps.}
From this and the finite propagation speed, we deduce that for $s_m/s_n\le t\le 1$
\EQ{
 \|\vec u_n(t)\|_{\HH(|x|>R\rho(s_m)/s_n+t-s_m/s_n)}\ll \eps.}
Letting $n\to \infty$, we infer that for $0<t\le 1$
\EQ{
 \|\vec u_n(t)\|_{\HH(|x|>t)}\ll \eps.}
Since $\eps>0$ is arbitrary, this implies that $u_\infty$ is supported on $|x|\le t$, as claimed.} This completes the proof of Proposition \ref{prop:Sg+ scatt}.
\end{proof}

In order to complete the proof of Theorem~\ref{thm: Main}, we now exhibit open data sets at time $t=0$ such that we have blow up/scattering at $t=\pm\infty$, four possibilities in all.
However, this has been done in the radial case \cite{CNLW1} by producing four solutions starting from the neighborhood $\HH_*\setminus\HH_X$ and exiting from it in finite time in both time directions, for all four combinations of $\Sg$ at the exiting times.
Since such behavior is obviously stable in the energy space $\HH$, we get an open set around each solution by the local wellposedness.

In fact, the initial data for such solutions can be given explicitly by
\EQ{ \label{sample}
 \vec u(0) = \vec W + \e \U a \rho, \pq \U a=\mat{\pm 1 \\ 0},\ \mat{0 \\ \pm 1}, \pq 0<\e\ll\e_*.}

For any solution $u$ in the region $d_W(\vec u(t))<\de_E$, Lemma \ref{lem:lambda dom} yields
\EQ{
 \vec u(t)=T^{c(t)}\vec S^{\s(t)}(\sg\vec W+\uv(t)), \pq \uv(t)=\U\la(t)\rho+\U\ga(t),}
with $d_W(t):=d_W(\vec u(t))\sim|\U\la|+\|\ga\|_E$, and
from the proof of Lemma \ref{lem: ejection},
\EQ{
 \pt \p_\ta\la_1=\la_2, \pq \p_\ta\la_2=k^2\la_1+O(d_W^2),
 \pr \p_\ta[\|\ga\|_E^2+O(\|\ga\|_Ed_W^2)]=O(\|\ga\|_Ed_W^2).}
In particular, if $\ga(0)=0$ and the linearized solution $\U\la^0$ for $\U\la$ satisfies for $0\le \ta'\le \ta$
\EQ{
 \int_0^{\ta'}|\U\la^0(\ta'')|d\ta'' \lec |\U\la^0(\ta')|<\de_E,}
then we deduce that (see \cite{NakS0,NakS} for more detail in the subcritical radial case)
\EQ{
 |\U\la-\U\la^0|+\|\ga\|_E \lec |\U\la^0|^2 \simeq d_W^2.}
This is the case for \eqref{sample}, with
\EQ{
 \CAS{\U a=(\pm 1,0) \implies \U\la^0=\pm\e(\cosh(k\ta),k\sinh(k\ta)), \\
 \U a=(0,\pm 1) \implies \U\la^0=\pm\e(\sinh(k\ta)/k,\cosh(k\ta)).}}
Moreover, when $\e\ll|\U\la^0|\ll\de_S$, the solution is in the region $\HH_X$ with $\Sg(\vec u)=-\sign\la^0_1$.
Hence the solution for $\U a=(\pm 1,0)$ respectively blows up and scatters both in $t<0$ and $t>0$, while the solutions for $\U a=(0,\pm 1)$ blows up in $\pm t>0$ and scatters for $t\to\mp\I$.
One can easily check that the former case $\U a=(\pm 1,0)$ is actually in the Kenig-Merle (or Payne-Sattinger) criterion $E(\vec u)<J(W)$ and $\pm K(u(0))<0$, while the latter case $\U a=(0,\pm 1)$ satisfies $E(\vec u)>J(W)$.

\section*{Note added in proof}
After submitting this paper, the authors learned the further progress by Duyckaerts, Kenig and Merle \cite{DKM4}, where they prove the soliton resolution conjecture for all solutions which are bounded in the energy space, albeit under the radial assumption. In particular, their result contains that under our  energy constraint and away from the ground states, there are {\it at most} four sets of global dynamics, namely $A_{\pm,\pm}$ in our notation.  In particular, \cite{DKM4} implies the remarkable property (in
the radial case and under our energy constraint) that all blowup outside of the soliton tube {\em is of type~$I$}.

We should emphasize, however, that their analysis is aimed at the asymptotic behavior around the boundary of the existence interval, while our analysis is concerned with the intermediate dynamics. Specifically, their analysis does not (at least immediately) yield the following facts proved in this paper:
\begin{enumerate}
\item The two sets of global dynamics $A_{\pm,\mp}$ have non-empty interior, namely stable sets of solutions with dynamical transition from the scattering to the norm blowup (and vice versa).
\item There is no solution under our energy constraint which stays close (or scatters) to $\cS$ around one endpoint of the existence time, and to $-\cS$ around the other endpoint.
\end{enumerate}
Note that our proof for the existence (1) crucially relies on the one-pass theorem, in which the non-existence (2) is included.
With higher energy, such solutions as in (2) can exist, but it is not deduced from \cite{DKM4} either, requiring some analysis on the intermediate dynamics.

\medskip
\quad
Received January 2012; revised August 2012.
\medskip

\end{document}